\documentclass[11pt]{amsart}

\usepackage{amsthm, amsfonts, amscd, graphicx, pinlabel}
\usepackage[all]{xy}
\usepackage{hyperref}
\PassOptionsToPackage{unicode}{hyperref}
\usepackage[usenames,dvipsnames]{xcolor}

\newcommand{\e}{\ensuremath{\eta}} 
\newcommand{\te}{\tilde{\e}}

\newcommand{\pd}[2]{\ensuremath{{\partial_{#2} #1}}}

\DeclareMathOperator{\ind}{Ind}

\def\TM+{T^*(\rr_+ \times M)}
\def\rp{\mathbb R_+}

\newcommand{\rr}{\ensuremath{\mathbb{R}}}
\newcommand{\zz}{\ensuremath{\mathbb{Z}}}

\newcommand{\nn}{\ensuremath{\mathbb{N}}}
\newcommand{\ff}{\ensuremath{\mathbb{F}}}

\theoremstyle{plain}
\newtheorem{thm}{Theorem}[section]
\newtheorem{cor}[thm]{Corollary}
\newtheorem{lem}[thm]{Lemma}

\newtheorem{claim}[thm]{Claim}
\newtheorem{prop}[thm]{Proposition}

\theoremstyle{definition}
\newtheorem{defn}[thm]{Definition}

\newtheorem*{ques}{Question}

\theoremstyle{remark}
\newtheorem{rem}[thm]{Remark}
\newtheorem{ex}[thm]{Example}

\numberwithin{equation}{section} 

\newcommand{\dfn}[1]{{\textbf {#1}}}

\newcommand{\rgh}[2]{\ensuremath{{GH}^{#1}(#2)}}

\newcommand{\leg}{\ensuremath{\Lambda}}

\DeclareMathOperator\Int{Int}
\DeclareMathOperator\im{Im}

\begin{document}

\title[Lagrangian Cobordisms via Generating Families]{Lagrangian Cobordisms via Generating Families:  Constructions and Geography} 

\date{\today}

\author[F. Bourgeois]{Fr\'ed\'eric Bourgeois} \address{Universit\'e Paris-Sud, 91405 Orsay, France} \email{bourgeois@math.u-psud.fr} \thanks{FB is partially supported by ERC Starting Grant StG-239781-ContactMath.}

\author[J. Sabloff]{Joshua M. Sabloff} \address{Haverford College,
Haverford, PA 19041} \email{jsabloff@haverford.edu} \thanks{JS is
partially supported by NSF grants DMS-0909273 and DMS-1406093.}

\author[L. Traynor]{Lisa Traynor} \address{Bryn Mawr College, Bryn
Mawr, PA 19010} \email{ltraynor@brynmawr.edu} \thanks{LT is
partially supported by NSF grant DMS-0909021.}

\begin{abstract} 
Embedded Lagrangian cobordisms  between Legendrian submanifolds 
 are produced from isotopy, spinning, and handle attachment constructions that employ the technique
 of generating families.  Moreover, 
any Legendrian with a generating family has an immersed Lagrangian filling with a compatible generating
family.  These constructions are applied in several directions, in particular to a non-classical geography question: any graded group satisfying a duality condition can be realized as the generating family homology of a connected
Legendrian submanifold in $\rr^{2n+1}$ or in the $1$-jet space of any compact $n$-manifold
with $n \geq 2$.
\end{abstract}

\maketitle


\section{Introduction}
\label{sec:intro}

Lagrangian cobordisms between Legendrian submanifolds have recently enjoyed increasing interest, partly due to their centrality in the Symplectic Field Theory  \cite{egh} (see \cite{bee, ekholm:rsft, ekholm:rsft-survey} for recent examples)  and wrapped Fukaya category frameworks (see \cite{abb-schw:cotangent, as:wrapped-floer, fss:wrapped-floer}), and partly due to their connections to smooth topology, including the computation of the slice genus of a knot \cite{chantraine}.

This paper extends the study of Lagrangian cobordisms in two ways:  first, we develop tools to construct Lagrangian cobordisms between Legendrian submanifolds in $1$-jet bundles using  the Morse-theoretic framework of generating families. We call the resulting Lagrangian cobordism $\overline{L}$ between the Legendrians $\leg_-$ and $\leg_+$ a \dfn{gf-compatible} cobordism and write $(\leg_-, f_{-})  \prec_{(\overline{L}, F)} (\leg_+, f_{+})$; see below for a precise definition. An important special case is  a gf-compatible null-cobordism $(\emptyset, f_-) \prec_{(\overline{L},F)} (\leg_+,f_+)$, which we call a \dfn{gf-filling}. Second, we investigate applications of our constructions, with an emphasis on solutions to geography questions for Legendrian submanifolds.

\subsection{Constructions}

We  study four constructions of Lagrangian cobordisms:
\begin{description}
\item[Immersed Fillings (see Proposition~\ref{prop:immersed-filling})] Every Legendrian submanifold with a generating family has an immersed gf-filling.
\item[Spinning (see Proposition~\ref{prop:front-spin})] An embedded, gf-compatible,  $n$-dimensional cobordism may be ``spun'' to an $(n+1)$-dimensional embedded gf-compatible cobordism.
\item[Legendrian Isotopy (see Proposition~\ref{prop:isotopy})] A Legendrian isotopy starting at a Legendrian with a generating family induces a gf-compatible embedded cobordism.
\item[Handle Attachment / Embedded Surgery (see Theorem~\ref{thm:q-surgery})] In an appropriate set-up, it is possible to create a gf-compatible cobordism by attaching a Lagrangian handle, thereby realizing an embedded Legendrian surgery.  This construction may also be performed locally to create a Lagrangian cobordism even without the presence of a global generating family.
\end{description}

The construction of an immersed filling is novel, but others have studied the spinning, isotopy, and surgery constructions. The fact that Legendrian isotopy induces a Lagrangian cobordism (though without the gf-compatibility) has previously appeared in \cite{chantraine} and \cite[Lemma 4.2.5]{eg:finite-dim}.
 The surgery construction and its relation to generating families is closely related to work of Entov \cite{entov:surgery}, though his work applies to a somewhat different setting. During the preparation of this paper, alternative approaches, without reference to generating families, to spinning \cite{golovko:tb, golovko:higher-spin} and surgery \cite{rizell:surgery, ehk:leg-knot-lagr-cob} were developed.
  
\subsection{Applications}

With these constructions in hand, we turn to their applications.  The first type of application involves finding Lagrangian fillings of Legendrian submanifolds in $J^1\rr^n$ using the surgery and isotopy constructions.  The $3$-dimensional case is particularly interesting, as Chantraine proved that a Lagrangian filling for a Legendrian knot realizes the smooth $4$-ball genus of the underlying smooth knot \cite{chantraine}.  In Section~\ref{sec:3d}, we demonstrate the existence of Lagrangian fillings for several  families of links:  Legendrian links that topologically are
twisted Whitehead doubles or $0$-closures of positive braids (see Propositions~\ref{prop:wh} and \ref{prop:braid}).  The techniques we discuss here also have interesting applications that have appeared elsewhere:
\begin{itemize}
\item Any positive knot has a Legendrian representative with a Lagrangian filling \cite{positivity};
\item There exist Legendrian links with non-homeomorphic fillings, answering a question of Boileau and Fourrier \cite{boileau-fourrier} about the uniqueness of fillings of links by complex curves \cite{polyfillability}; and 
\item For any Legendrian knot, it is possible to construct a Lagrangian cobordism so that $\leg_-$ is a Legendrian unknot; for twist, torus,  and low crossing knots  this can be done so that the cobordism has minimal smooth genus.   \cite{bty}.
\end{itemize}
We note that Ekholm, Honda, and K\'alm\'an's construction of Lagrangian cobordisms \cite{ehk:leg-knot-lagr-cob} (see also \cite{rizell:surgery}) would work equally well for these applications.  In contrast, there are also several applications of the fact that the cobordisms we construct in this paper are gf-compatible:
\begin{itemize}
\item There exist Legendrian spheres in $J^1\rr^n$ that have arbitrarily many different generating family homologies \cite{polyfillability}; and 
\item Denote the space of all Legendrian spheres in $J^1\rr^n$ by $\mathcal{L}^n$.  For any $n > 1$, there exists a Legendrian sphere $\leg^n$ so that $\pi_1(\mathcal{L}^n; \leg^n)$ is nontrivial \cite{ss:pi-k}.
\end{itemize}

The second type of application is to  geography questions for Legendrian submanifolds:  

\begin{quote}
\emph{Which collections of invariants can be realized by a Legendrian submanifold?}  
\end{quote}

We first restrict the geography question to the generating family cohomology invariant, which is defined in Section~\ref{sec:background}, below.  Denote the Poincar\'e polynomial of the generating family cohomology $GH^*(f)$ by $\Gamma_f(t)$.

\begin{ques}[Non-Classical Generating Family Geography]
  Given a Laurent polynomial $P(t) \in \nn[t, t^{-1}]$, is there a connected Legendrian submanifold $\leg \subset J^1M$ (where $M$ is compact or equal to $\rr^n$) with a linear-at-infinity generating family $f$ so that $\Gamma_f(t) = P(t)$?
\end{ques}

Any investigation of geography must begin by understanding obstructions.  The first obstruction to a polynomial being the Poincar\'e polynomial for the generating family homology is the duality of \cite{josh-lisa:obstr}, which
we strengthen below in Theorem~\ref{thm:duality}; see also \cite{high-d-duality, f-r, duality}.  In particular, duality implies the following:

\begin{thm} \label{thm:compatible}
	Given a Legendrian $\leg \subset J^1M$ with a linear-at-infinity generating family $f$, the generating family Poincar\'e polynomial is of the form
	\begin{equation} \label{eqn:poly}
\Gamma_f(t) = (q_0 + q_1t + \dots + q_nt^n) + p(t) + t^{n-1}p(t^{-1}),
\end{equation}
where $p(t) = \sum_{i \in \mathbb Z, i \geq \lfloor \frac{n-1}{2} \rfloor} p_i t^i$,
$q_k + q_{n-k}$ is the $k^{th}$ Betti number of  $\leg$, and $q_n \neq 0$.  
\end{thm}
 
  For  connected Legendrians, the polynomial in  (\ref{eqn:poly}) has $q_n = 1$ and $q_0 = 0$. A polynomial of the form of the right hand side of equation (\ref{eqn:poly}) with $q_n \neq 0$ is \dfn{compatible with duality}; if, in addition, the polynomial has $q_n = 1$ and  $q_0 = 0$, then the polynomial
is in  \dfn{connected form}. 
The spinning, isotopy, and embedded surgery constructions, coupled with the Cobordism Exact Sequence of \cite{josh-lisa:obstr} (see below), allow us to prove the following complete answer to the non-classical generating family geography question:

\begin{thm}  \label{thm:geography}
  If the Laurent polynomial $P(t)$ is compatible with duality and is in connected form,  
   then for any $n\geq 2$, there exists a connected Legendrian submanifold of $J^1M$ with a generating family $f$ so that $\Gamma_f(t) = P(t)$.  
\end{thm}

The $n=1$ case was proven by Melvin and Shrestha for Legendrian contact homology \cite{melvin-shrestha}, and the work of Fuchs and Rutherford \cite{f-r} implies that Melvin and Shrestha's results also hold for generating family homology.

In fact, the Legendrians constructed to prove the theorem above are all gf-compatibly Lagrangian cobordant to higher-dimensional analogues of the Hopf link, and (hence) are all Lagrangian null-cobordant (though not necessarily gf-compatibly).  Our techniques are suitably functorial that, using the results of \cite{high-d-duality} in place of Theorem~\ref{thm:compatible} to restrict the admissible polynomials and \cite{rizell:lifting, golovko:tb} for Legendrian contact homology versions of the Cobordism Exact Sequence, Theorem~\ref{thm:geography} also holds for Legendrian contact homology.

Our investigations into the non-classical geography question also yield results for the classical Thurston-Bennequin and rotation number invariants in higher dimensions:

\begin{ques}[Classical Fillable Geography]
	Given a pair of integers $(\tau, \rho)$, does there exist a fillable Legendrian $n$-sphere $\leg \subset J^1\rr^n$ whose Thurston-Bennequin and rotation numbers are $\tau$ and $\rho$, respectively?
\end{ques}

It is well-known that if $\leg$ is fillable, then its rotation number vanishes; see the discussion in \cite[Section 2.2]{chantraine}, for example.  Further, for even-dimensional Legendrians, the Thurston-Bennequin number is determined by the Euler characteristic; see \cite{ees:high-d-geometry} or \cite[Appendix A]{murphy:loose}, for example.  Thus, the only interesting question for classical fillable geography involves the Thurston-Bennequin number in odd dimensions, where $tb$ must be odd \cite[Appendix A]{murphy:loose}.  We provide a complete answer to the classical fillable geography question in Subsection 6.3:

\begin{thm} \label{thm:classical-fillable}
	For any odd $n$ and $\tau$, there exists an odd-dimensional fillable Legendrian sphere $\leg \subset J^1\rr^n$  with $tb(\leg) = \tau$.
\end{thm}

\subsection*{Plan of the Paper}
\label{ssec:plan}

We set notation and sketch the necessary background for Lagrangian cobordisms and generating families in Section~\ref{sec:background}.  In Section~\ref{sec:basic}, we work in the setting of gf-compatible Legendrians and Lagrangians to construct immersed Lagrangian fillings, spun Legendrian submanifolds and Lagrangian cobordisms, and Lagrangian cobordisms arising from Legendrian isotopy.  The construction of the attachment of a Lagrangian handle is described in Section~\ref{sec:surgery}.  Applications of the constructions begin in Section~\ref{sec:3d} with several examples of gf-fillings for Legendrian knots in $J^1\rr$. Finally, in Section~\ref{sec:geo-bot}, we use constructions developed in this paper, together with the Cobordism Exact Sequence of \cite{josh-lisa:obstr}, to investigate the  geography questions mentioned above.

\subsection*{Acknowledgements}

The authors thank the American Institute of Mathematics, the Universit\'e de Nantes, the Royal Academies for Science and the Arts of Belgium, and the Banff International Research Station for hosting conferences at which the authors initiated and completed the work discussed in this paper. The authors also thank Matt Hedden for his help in straightening out some references.

\section{Background Notions}
\label{sec:background}

In this section, we briefly review the language of generating families for Legendrian submanifolds of $J^1M$ and for Lagrangian cobordisms between them.  See \cite{f-r, lisa-jill, josh-lisa:obstr, lisa:links} for the original definitions and for more details.

\subsection{Generating Families for Legendrians}
\label{ssec:gf}

The technique of generating families may be used to construct Legendrian submanifolds of the $1$-jet bundle of a smooth manifold $M$.  For future reference, we denote the \dfn{front projection} by $\pi_{xz}: J^1M \to J^0M$ and the \dfn{base projection} by $\pi_x: J^1M \to M$.

Given a smooth manifold $M$, let $f: M^n \times \rr^N \to \rr$ be a smooth function, where $M \times \rr^N$ has coordinates $(x,\e)$. Unless otherwise noted, we assume that $M$ is either compact or $\rr^n$.
Suppose that $\mathbf{0}$ is a regular value of the map $\partial_{\e} f: M \times \rr^N \to \rr^N$.   We define the \dfn{fiber critical set} of $f$ to be the $n$-dimensional submanifold $\Sigma_f = (\partial_\eta f)^{-1}(\mathbf{0})$.  Define immersions $i_f: \Sigma_f \to T^*M$ and $j_f: \Sigma_f \to J^1M$ in local coordinates by:
\begin{align*}
  i_f(x,\e) &= (x, \partial_x f(x,\e)),\\
  j_f(x,\e) &= (x, \partial_x f(x,\e), f(x,\e)).
\end{align*}
The image $L$ of $i_f$ is an immersed Lagrangian submanifold; the image $\leg$ of $j_f$ is an immersed Legendrian submanifold.  We say that $f$ \dfn{generates} $L$ and $\leg$, or that $f$ is a \dfn{generating family (of functions)}.

Not every Legendrian submanifold has a generating family --- see \cite{chv-pushkar}, for example --- but for those that do, the Morse theory of the set of generating families gives rise to interesting non-classical invariants. Since the domain of a generating family $f: M \times \rr^N \to \rr$ is not compact, we need to control the behavior of $f$ at infinity.  We define $f$ to be \dfn{linear-at-infinity} if there exists a non-zero linear function $A: \rr^N \to \rr$ such that $f(x,\e) = A(\e)$ outside a compact subset of $M \times \rr^N$.

The Morse-theoretic invariant that a generating family attaches to a Legendrian submanifold is defined using the \dfn{difference function} $\delta: M \times \rr^N \times \rr^N \to \rr$:
\begin{equation}
  \delta(x, \e, \te) = f(x,\te) - f(x,\e).
\end{equation}
It is not hard to show that there is a one-to-one correspondence between the Reeb chords of $\leg$ (i.e. line segments parallel to the $z$ axis that begin and end on $\leg$) and the critical points of $\delta$ with positive critical value. 
Choose  $\omega > \epsilon> 0$ so that all positive critical values of $\delta_f$ lie between $\omega$ and $\epsilon$, and define the \dfn{relative (resp.\ total) generating family cohomology} of $f$ to be the relative cohomology of the $\omega$ and $\epsilon$ (resp.\ $-\epsilon$) sublevel sets of $\delta$:
\begin{align*}
GH^k(f) &= H^{k+N+1}(\delta^\omega, \delta^\epsilon) \\ \widetilde{GH}\hphantom{}^k(f) &= H^{k+N+1}(\delta^\omega, \delta^{-\epsilon}).
\end{align*}
The \dfn{generating family homology} is defined analogously with the same index shift.  In most other sources, the relative generating family cohomology is simply called the generating family cohomology, and we will adopt this convention throughout this paper except in the appendix.  The long exact sequence of the triple $(\delta^\omega, \delta^\epsilon,\delta^{-\epsilon})$, together with the Thom isomorphism, relates the relative and total generating family cohomologies:
\begin{equation} \label{eq:gh-les}
	\cdots \to GH^k(f) \to \widetilde{GH}\hphantom{}^k(f) \to H^{k+1}(\leg) \to \cdots.
\end{equation}

The set of all generating family cohomology groups taken over all possible generating families for $\leg$ forms an invariant of $\leg$ up to Legendrian isotopy \cite{f-r, lisa:links}.

\subsection{Generating Families for Lagrangian Cobordisms}
\label{ssec:lagr}

Let $(X,\alpha)$ be a contact manifold and denote its symplectization by $(\rr \times X, d(e^t\alpha))$. A Lagrangian submanifold $\overline{L}$ of the symplectization is a \dfn{Lagrangian cobordism} between the Legendrian submanifolds $\leg_\pm \subset X$ if there exists $t_\pm>0$ such that:
\begin{align*}
  \overline{L} \cap\left( (-\infty, -T] \times X \right) &= (-\infty, t_-] \times \leg_-, \\
  \overline{L} \cap \left( [T, \infty) \times X \right)&= [t_+, \infty) \times \leg_+.
\end{align*}
We denote such a Lagrangian cobordism by $\leg_- \prec_{\overline{L}} \leg_+$, and we denote the compact manifold $\overline{L} \cap [t_-,t_+]$ by $L$.

If the contact manifold is a $1$-jet bundle $J^1M$, then there is a symplectomorphism between its symplectization and $T^*(\rr_+ \times M)$ with its canonical symplectic structure.
As in  \cite[Section 4]{josh-lisa:obstr},  we shift our perspective to $T^*(\rr_+ \times M)$ so that we may use generating families to describe Lagrangian cobordisms.  Lagrangians constructed through generating families will be exact and these will map to exact Lagrangians in the symplectization of $J^1M$.  For ease of notation, we denote a Lagrangian cobordism $\overline{L}$ and its image in $T^*(\rr_+ \times M)$ by the same symbol.

If the functions $f_\pm: M \times \rr^N \to \rr$ and $F: \rr_+ \times M \times \rr^N \to \rr$ satisfy the following relation for some $t_-<t_+$:
\begin{equation}
	F(t,x,\e) = \begin{cases} tf_-(x,\eta) & t \leq t_-\\
		tf_+(x,\eta) & t \geq t_+
	\end{cases},
\end{equation}
then we say that the triple $(F,f_-,f_+)$ is \dfn{compatible}.  A \dfn{gf-compatible} Lagrangian cobordism consists of a Lagrangian cobordism $\leg_- \prec_{\overline{L}} \leg_+$ together with a compatible triple of generating families for the three objects involved.  We denote a gf-compatible cobordism by
$$(\leg_-,f_-) \prec_{(\overline{L},F)} (\leg_+,f_+);$$
for a filling, i.e. when $\leg_- = \emptyset$, we will frequently use the shorthand notation $\emptyset \prec_{(\overline{L},F)} (\leg, f)$, with the understanding
 that $f_-$ is a linear function.
 We require the triple of compatible functions to be \dfn{tame} in the sense that $f_\pm$ are linear-at-infinity and $F$ is \dfn{slicewise linear at infinity}, i.e. each $F(t,\cdot,\cdot)$ is equal to a linear function $A_t(\e)$ outside a compact set of $\{t\} \times M \times \rr^N$.  

A key finding of \cite{josh-lisa:obstr} is the following:

\begin{thm}[Cobordism Exact Sequence] \label{thm:cobord-les}  
A gf-compatible Lagrangian cobordism $(\leg_-, f_-) \prec_{(\overline{L}, F)} (\leg_+, f_+)$ induces a linear map 
$\Psi_F: GH^k(f_-) \to GH^k(f_+)$ that fits into the following long exact sequence:
\begin{equation} \label{eq:cobord-les}
  \xymatrix@1{
    \cdots \ar[r] & \rgh{k}{f_-} \ar[r]^{\Psi_F} & \rgh{k}{f_+} \ar[r] & H^{k+1}(L, \leg_+) \ar[r] & \cdots.
  }
\end{equation}
\end{thm}

\section{Basic Constructions}
\label{sec:basic}

In this section, we discuss three global constructions of gf-compatible Lagrangian cobordisms.  The first is an explicit proof of the Gromov-Lees theorem in this setting, namely that a Legendrian submanifold with a linear-at-infinity generating family has a gf-compatible, immersed Lagrangian filling.  The second is a translation of Ekholm-Etnyre-Sullivan's spinning construction \cite{ees:high-d-geometry} to the generating family setting, together with a generalization to properly embedded Lagrangian submanifolds in half-spaces; see also \cite{golovko:higher-spin}.  The final construction is a translation to the generating family setting of Chantraine's proof that Legendrian isotopy induces a Lagrangian cobordism \cite{chantraine}; see also \cite{eg:finite-dim} and \cite{golovko:tb}.  

\subsection{Immersed Lagrangian fillings}
\label{ssec:immersed}

If a Legendrian knot $\leg \subset \rr^3$ has a generating family, then $r(\leg) = 0$ (see \cite{f-r}, for example) and thus has an immersed Lagrangian filling by work of Chantraine \cite[Remark 4.2]{chantraine}. We strengthen this statement by
extending it to Legendrian submanifolds in arbitrary dimensions and by showing that the filling can be constructed via a compatible generating family.  The idea is to smoothly deform a generating family to a linear function.  
  
\begin{prop}\label{prop:immersed-filling}  
  If $\leg \subset J^1M$ is a Legendrian submanifold with a linear-at-infinity generating family $f$, then there exists an \emph{immersed} $gf$-compatible cobordism $\emptyset \prec_{(\overline{L},F)} (\leg,f)$. 
  \end{prop}

\begin{proof} Let $f: M \times \rr^N \to \rr$ be a linear-at-infinity generating family for $\leg$; assume that $f$ agrees with
the nonzero linear function $A(\e)$ outside a compact set.  To show the existence of a gf-compatible immersed filling, it suffices to construct  $F: \rp \times M \times \rr^N \to \rr$ so that 
  \begin{enumerate}  
   \item  For all $t$ outside a compact set of $\{t\} \times M \times \rr^N$, $F(t,x,\e)$ is  a nonzero linear function $B_t(\e)$ ;
  \item There exists a $t_-$ so that, for $t \leq t_-$,  $F(t, x, \e) = B_{t}(\e)$;
  \item There exists a $t_+$ so that, for $t \geq t_+$,   $F(t, x, \e) = t f(x, \e)$; and
  \item $\mathbf{0} \in \rr^N$ is a regular value of $\pd{F}{\eta}$.
  \end{enumerate}
  
To begin the process, choose a smooth, increasing function $\sigma: \rp \to \rr$ that is $0$ on $(0,1]$ and $1$ on $[2,\infty)$.     Define a function $G: \rp \times M \times \rr^N \to \rr$ by:
\begin{equation*}
  G(t,x,\e) = t\cdot \big( \sigma(t) f(x,\e) + (1-\sigma(t)) A(\e) \big).
\end{equation*}  
Then, for fixed $t_- < 1$ and $t_+ > 2$,  $G$ satisfies  conditions (1--3) desired for $F$.  For later purposes, note that outside a compact
set of $\{t \} \times M \times \rr^N$ and for $t \leq t_-$, $G(t,x,\e) = t A(\e)$.

  We now modify $G$ to guarantee condition that (4) is satisfied.  Since $f$ is a generating family, $\mathbf{0} \in \rr^N$ is a regular value of  $\pd{f}{\eta}$.  Since $f$ is linear-at-infinity, we may find an open, convex ball $U \subset \rr^N$ around $\mathbf{0}$ that consists of  regular values of $\pd{f}{\e}$.  By Sard's Theorem,  there exists $\varepsilon_G \in U$ that is a regular value of $\pd{G}{\e}$ and satisfies for all $t \in [1,2]$, as functions $\varepsilon_G \cdot \e \neq  tA(\e)$.  Choose a smooth path $\varepsilon: \rp \to U$ so that 
  $\varepsilon(t) \equiv \varepsilon_G$ when $t \leq  2$,   $\varepsilon(t) \cdot \e \not\equiv t A(\e)$ when $t \in [2, t_+]$, and
   $\varepsilon(t) \equiv 0$ when $t \geq  t_+$.  Now consider the function
  \begin{equation*}
  \begin{split}
  F(t, x, \e) &= G(t, x, \e) - \varepsilon(t) \cdot \eta \\
  &  = 
  \begin{cases}
  G(t,x, \e) - \varepsilon_G \cdot \e, &t \leq 2 \\
  tf(x,\e) - \varepsilon(t) \cdot \e, & t \in [2, t_+] \\
  tf(x, \e), & t \geq  t_+.
  \end{cases}
  \end{split}
  \end{equation*}

  Conditions (1--3) are still satisfied by $F$.  Further, since ${\varepsilon_G}$ is a  regular value of $\pd{G}{\e}$, it follows that $\mathbf{0}$ is a  regular value of $\pd{F}{\eta}$ when $t \leq 2$.  It remains to show that $\mathbf{0}$ is a regular value of $\pd{F}{\eta}$ when $t \geq  2$.  When $t \geq 2$, $\pd{F}{\e}(t, x, \e)$ vanishes if and only if $\pd{f}{\e} (x, \e) = \frac{\varepsilon(t)}{t}.$  Since $\varepsilon(t)$ lies in the convex set of regular values $U$ and $t \geq 2$,  $\frac{\varepsilon(t)}{t} \in U$ is a regular value of $\pd{f}{\e}$  for all $t \geq 2$.  Thus when $t \geq 2$,  at a point $(t, x, \e)$ where $\pd{F}{\e} = 0$,  we see that the final $n + N$  columns of the $N \times (1 + n + N)$ matrix $D(\pd{F}{\e})$  form a submatrix of rank $N$, as desired.  
\end{proof}

\subsection{Spinning Constructions}
\label{ssec:spinning}

Let $H^n$ denote the closed upper half space of $\rr^n$, i.e.\
$$H^n = \left\{ \mathbf{x} \in \rr^n\;:\; x_n \geq 0 \right\}.$$
Ekholm, Etnyre, and Sullivan \cite{ees:high-d-geometry} described a method of producing an $(n+1)$-dimensional Legendrian submanifold of $J^1\rr^{n+1}$ from an $n$-dimensional Legendrian in $J^1H^{n}$ by spinning a front diagram about the $z$ axis.  Ekholm and K\'alm\'an generalized the construction to twist-spun Legendrians \cite{ek:isotopy-tori}.  In this section, we generalize these constructions to certain properly embedded Legendrian submanifolds of $J^1H^{n}$, show how these constructions can be performed using generating families, and discuss spinning Lagrangian cobordisms in the spirit of Golovko \cite{golovko:tb}.

A smooth, properly embedded Legendrian submanifold $\leg$ of $J^1H^n$ is a \dfn{spinnable Legendrian} 
if for all $p \in \partial \leg \subset \partial J^1H^n$, the local parameterization of $\leg$ near $p$,  $\phi: H^n \to J^1(H^n)$, extends  to a smooth
map $\phi: \rr^n \to J^1(\rr^n)$ by
$\phi(x_1, \dots, x_n)  = \phi(x_1, \dots, -x_n)$ when $x_n < 0$.
A spinnable Legendrian $\leg$ gives rise to a \dfn{spun Legendrian} $\leg^s \subset J^1 \rr^{n+1}$ whose front projection is
obtained by  rotating the front projection  of $\leg$ about the subspace $\{x_n = x_{n+1} = 0\}$. More generally, 
a \dfn{spinnable Legendrian loop} consists of a smooth isotopy $\leg_\theta$, 
$\theta \in S^1 = [0, 2\pi]/\sim$, of spinnable Legendrians  so that for all $\theta$, a neighborhood of $\partial \leg_\theta$ does not vary
with respect to $\theta$. 
 A spinnable isotopy $\leg_\theta$, in turn, gives rise to a \dfn{twist-spun Legendrian} $\leg^{ts} \subset J^1\rr^{n+1}$ by following the isotopy during a rotation about $\{x_n = x_{n+1} = 0\}$. 
 More specifically, if $\phi_\theta: U \to H^n \times \rr $ are smooth, local parameterizations of the fronts  $\leg_\theta$ given by
$$\phi_\theta(q) = \left( x_1^\theta(q), \ldots, x_n^\theta(q), z^\theta(q)\right),$$ 
then a parameterization of  the front of the twist-spun submanifold $\leg^{ts}$ is given by
$$\phi^{ts}(q, \theta) = \left(x_1^\theta(q), \ldots, x_{n-1}^\theta(q), \ x_{n}^\theta(q) \cos \theta, \ x_{n}^\theta(q) \sin \theta, \  z^\theta(q) \right).$$

 We say that  a generating family
  $f: H^n \times \rr^N \to \rr$ is  a \dfn{spinnable generating family}   if $f$ has a smooth extension to $f: \rr^n \times \rr^N \to \rr$ with
  $f(x_1, \dots, x_n, \e) = f(x_1, \dots, -x_n, \e)$, for $x_n < 0$.   A \dfn{spinnable loop of generating families} consists of a smooth path
  of spinnable generating families $f_\theta$, $\theta \in S^1 = [0, 2\pi]/ \sim$, 
  so that in a neighborhood of $\partial H^n \times \rr^N$, $f_\theta$ does not depend on $\theta$.    
  By construction, spinnable (loops of) generating families generate spinnable (loops of) Legendrians.
  A \dfn{spinnable loop of linear-at-infinity
 generating families} consists of a spinnable loop of generating families $f_\theta$ so that outside a compact set of $H^n \times \rr^n$, 
 $f_\theta(x, \eta)$ agrees with a nonzero linear function $A(\eta)$ for all $\theta$.

\begin{prop} \label{prop:front-spin}
  If a Legendrian submanifold $\leg \subset J^1H^n$ has a spinnable
  linear-at-infinity 
  generating family then the spun Legendrian submanifold $\leg^s$ has a  linear-at-infinity 
  generating family.  More generally, if $\leg_\theta$ is generated by a  spinnable loop of  tame linear-at-infinity 
  generating families 
  then the twist-spun Legendrian submanifold $\leg^{ts}$ has a 
  linear-at-infinity 
  generating family.
\end{prop}

\begin{proof}
  The first part of the claim follows from the second via the constant isotopy, so we only prove the twist-spun claim.  If $f_\theta: H^n \times \rr^N \to \rr$ 
  is a spinnable loop of linear-at-infinity generating families for  $\leg_\theta$, then we claim that the function $f^{ts}: \rr^{n-1} \times \rr^2 \times \rr^N \to \rr$
  given by
 $$f^{ts}(x_1, \ldots, x_{n-1}, x_n, \theta, \e) = f_{\theta}(x_1, \ldots, x_{n-1}, x_{n}, \e),$$
 where $(x_n, \theta)$ denote polar coordinates on $\rr^2$, 
is linear-at-infinity and generates $\leg^{ts}$.

Since $f_\theta|_{\partial H^n \times \rr^N}$ does not depend on $\theta$, $f_\theta(x_1, \dots,x_{n-1}, 0, \eta)$ is independent of
$\theta$ and so  $f^{ts}$ is well-defined.  The condition that all $f_\theta$ coincide on a neighborhood of
$\partial H^n \times \rr^N$ guarantee that $f^{ts}$ is smooth.  Since each $f_\theta$ agrees with the same
linear function $A(\e)$ outside a compact set, $f^{ts}$ is linear-at-infinity.   A straightforward calculation shows that since $\mathbf{0}$ is
a regular value of $\pd{f_\theta}{\e}$ for all $\theta$, $\mathbf{0}$ is a regular value of $\pd{f^{ts}}{\e}$, and thus $f^{ts}$ is a generating family for $\leg^{ts}$.     
\end{proof}

\begin{ex}
  One way to create higher-dimensional analogues of the Legendrian unknot is to cut the unknot in half along the $z$ axis and spin the result; see Figure~\ref{fig:front-spin}.  Iterating this construction yields ``flying saucers'' with generating families in all dimensions.
\end{ex}

\begin{figure}
  \centerline{\includegraphics{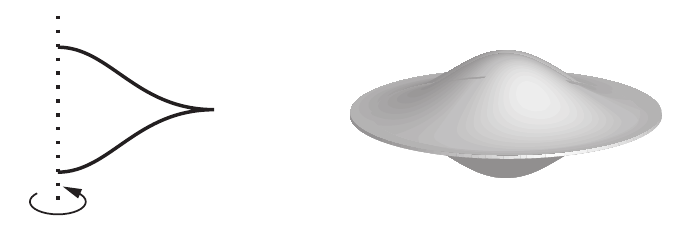}}
  \caption{Spinning half of a Legendrian unknot yields a ``flying saucer''. }
  \label{fig:front-spin}
\end{figure}

\begin{rem}
  In \cite{kalman:mono1}, K\'alm\'an constructed a loop of Legendrian trefoil knots --- i.e. a spinnable isotopy --- that is not contractible in the space of Legendrian trefoils; this isotopy was viewed in the context of spinning in \cite{ek:isotopy-tori}.  Even though the trefoil has a generating family, and hence the isotopy gives rise to a path of generating families for the trefoils along the isotopy, the fact that the twist-spun Legendrian torus does not have an augmentation
  shows that this path of generating families is not a loop.
\end{rem}

The spinning construction for generating families may be extended to Lagrangian cobordisms; see \cite{golovko:tb} for a relative of this construction that does not consider generating families. 
 As above, a \dfn{spinnable Lagrangian cobordism} is a properly embedded Lagrangian cobordism $\overline{L} \subset T^*(\rr_+ \times H^n)$ so that
 for all $p \in \partial \overline{L} \subset \partial T^*(\rr_+ \times H^n)$, the local parameterization of $\overline{L}$ near $p$,  $\phi: \rr \times H^n \to T^*(\rr_+ \times H^n))$, extends  to a smooth
map $\phi: \rr \times \rr^n \to   T^*(\rr_+ \times \rr^n)  $ by
$\phi(t, x_1, \dots, x_n)  = \phi(t, x_1, \dots, -x_n)$.
We say that  a generating family
  $F: \rp \times H^n \times \rr^N \to \rr$ is  a \dfn{spinnable generating family}   if $F$ has a smooth extension to $F: \rp \times \rr^n \times \rr^N \to \rr$ by
  $F(t, x_1, \dots, x_n, \e) = F(t, x_1, \dots, -x_n, \e)$.  By construction, a spinnable generating family generates a spinnable Lagrangian cobordism.

 The construction of a spun Lagrangian cobordism with a spun generating family then follows exactly the same steps.

\begin{prop} \label{prop:cob-spin}   
  Given a spinnable Lagrangian cobordism $\overline{L}$ with 
  tame, spinnable generating family $F$ of the form $(\leg_-, f_-) \prec_{(\overline{L}, F)} (\leg_+, f_+)$ in $T^*(\rr_+ \times H^n)$, the spun Lagrangian cobordism $\leg_-^s \prec_{\overline{L}^s} \leg_+^s$ has a compatible, tame
  generating family.
\end{prop}

\subsection{Legendrian Isotopy}
\label{ssec:isotopy}

In this subsection, we translate the fact that a Legendrian isotopy induces a Lagrangian cobordism \cite{chantraine, eg:finite-dim, golovko:tb} to the generating family setting.  Namely, we prove:

\begin{prop} \label{prop:isotopy}
  Suppose that $\leg_-$ is a Legendrian submanifold of $J^1M$ with a tame
    generating family $f_-$
   and that $\leg_-$ is Legendrian isotopic to $\leg_+$.  Then there exists an embedded gf-compatible Lagrangian cobordism $(\leg_-,f_-) \prec_{(\overline{L},F)} (\leg_+,f_+)$.  Further, if $f_-$ is tame,
   then $F$ is tame. 
\end{prop}

The proof of Proposition~\ref{prop:isotopy} will use the notion of a difference function $\delta(x, \e, \te)$, as introduced in Section~\ref{ssec:gf}. The following lemma is our key technical tool:

\begin{lem}[Family of Functions Construction] \label{lem:emb-family}
  Given a $1$-parameter family of functions $f_t: M \times \rr^N \to \rr$, $t \in \rr_+$, define the function $F: \rr_+ \times M \times \rr^N \to \rr$ by
$$F(t,x,\e) = t\,f_t(x,\e).$$  If
\begin{enumerate}
\item $\mathbf{0}$ is a regular value of $\partial_\e F$, and 
\item for all $(x, \e, \te)$ with $\e \neq \te$ in the fiber critical set of the difference function $\delta_t(x, \e, \te) = f_t(x,\te) -f_t(x, \e)$, we have:
\begin{equation}
\frac{\partial}{\partial x} \delta_t(x, \e, \te) = 0 \implies \delta_t(x, \e, \te) \neq -t \frac{\partial}{\partial t} \delta_t(x, \e, \te),
\end{equation}
\end{enumerate}
then $F$ generates an \emph{embedded} Lagrangian submanifold of $T^*(\rr_+ \times M)$.
\end{lem}

\begin{proof}  If $\mathbf{0}$ is a regular value of $\pd F{\e}$, then $F$ generates an
  immersed Lagrangian $\overline{L} \subset T^*( \rp \times M)$ given by
  \begin{equation} \label{eq:lagrangian}
    \overline{L} =  \left\{ \left(t, x, f_t(x, \e) + t \pd{f_t}{t} (x, \e), \, t \pd{f_t}{x}(x, \e) \right) : \pd{f_t}{\e}(x, \e) = 0 \right\}.
  \end{equation}
  A direct calculation shows that the double points of $\overline{L}$ are in bijective correspondence with  points $(t, x, \e, \te)$ with $\e \neq \te$ 
  satisfying:
  \begin{enumerate} 
  \item  $(x, \e, \te)$ is in the fiber critical set of $\delta_t$, 
    \item$\frac{\partial}{\partial x} \delta_{f_t}(x, \e, \te)  = 0$,     and
  \item  $\delta_{f_t}(x, \e, \te) = -t \frac{\partial}{\partial t} \delta_{f_t}(x, \e, \te)$. 
    \end{enumerate} 
  By  hypothesis, it is impossible to simultaneously satisfy these three conditions, and thus the Lagrangian generated by $F$ is embedded.
\end{proof}

We are now ready to prove Proposition~\ref{prop:isotopy}.

\begin{proof}[Proof of Proposition~\ref{prop:isotopy}]
  Let $\leg_t$ be a $1$-parameter family of Legendrian submanifolds of $J^1M$ so that 
$\leg_t = \leg_+$ for $t \geq t_+$ and $\leg_t = \leg_-$ for $t \leq t_-$.  By the persistence of linear-at-infinity 
generating families under Legendrian isotopy 
(see, for example, \cite{chv:quasi-fns, lisa-jill}), we obtain a $1$-parameter family of 
linear-at-infinity generating families $f_t: M \times \rr^N \to \rr$ that generate $\leg_t$ for $t \in [t_-,t_+]$.  Extend $f_t$ to a smooth family for all $t \in \rr_+$ so that outside of a compact interval, we have $f_t = f_\pm$.

We will now check that, for an appropriate parametrization of the family $f_t$, $F(t,x,\e) = t\, f_{t}(x,\eta)$ satisfies the hypotheses of Lemma~\ref{lem:emb-family}, and hence generates the desired Lagrangian cobordism.    
First, it is straightforward to verify that since $\mathbf{0}$ is a regular value of $\partial_\e f_t$ for all $t$, $\mathbf{0}$ is also a regular value of $\partial_\e F$.
Second, notice that 
$\frac{\partial}{\partial x} \delta_{t}(x, \e, \te) = 0$  
when $(x,\e)$ and $(x,\te)$ correspond to the endpoints of a Reeb chord of $\leg_t$; the ``length" of the Reeb chord is precisely $\delta_{t}(x,\e,\te)$, 
which may be positive or negative.  Let $h>0$ denote the minimum absolute value of the lengths of Reeb chords of all of the Legendrians in the isotopy $\leg_t$.  It then suffices to show that for  points $(x, \e, \te)$ in the fiber critical set
of $\delta_{t}$, 
 we have
\begin{equation} \label{eq:isotopy} \frac{h}{t} >   | \partial_t \delta_{t} (x, \e, \te)| .
\end{equation}
 Since, for each $t$, the fiber critical set of $\delta_{t}$ is compact, and 
  $\partial_t \delta_{t} = 0$ for $t$ outside of a compact interval,
  it follows that $\partial_t \delta_{t}$ is bounded on the domain of interest.
   Thus, after an orientation-preserving diffeomorphism $\rho$ of $\rr_+$, we may assume that
   $\widetilde F(t, x, \e) = t \, f_{\rho(t)} (x, \e)$ 
   will satisfy Equation (\ref{eq:isotopy}).  
\end{proof}

We end this section with several remarks arising from the proof of Proposition~\ref{prop:isotopy} and its consequences.

\begin{rem}
	Since the Lagrangian cobordism generated in the proof of Proposition~\ref{prop:isotopy} is, in fact, a concordance (i.e.\ is diffeomorphic to $\leg \times \rr$), the Cobordism Exact Sequence in Theorem~\ref{thm:cobord-les} tells us that the cobordism map $\Psi_F$ is an isomorphism.
\end{rem}

\begin{rem}  It is important to point out that if a Lagrangian is constructed with the generating family $F(t, x, \e) = t f_t(x, \e)$, where $f_t$
generates a Legendrian $\leg_t$, the ``slices" of
the Lagrangian  are not, in general, $\leg_t$.  
An examination of
 Equation~(\ref{eq:lagrangian}) shows that if $f_t(x, \e)$ generates the Legendrian $\leg_t$, the corresponding $s$-slice, $s \in \rr$, of the Lagrangian in 
the symplectization $\rr \times J^1M$ generated
by $F(t, x, \e, \te)$  
 will agree with $\leg_t$ if and only if for some neighborhood of $t$, $f_t$ is constant with respect to $t$.  When
 $f_t$ is not constant, the corresponding slice of the Lagrangian differs from $\leg_t$ by a contribution
 of $t \pd{}{t} f_t(x, \e)$ to the $z$-coordinate of $\leg_t$.   
\end{rem}

\section{Attaching Lagrangian Handles}
\label{sec:surgery}

In this section, we explain how an embedded $(q-1)$-surgery on a Legendrian $\leg_- \subset J^1M$ with a generating family induces a gf-compatible Lagrangian cobordism given by attaching a $q$-handle.  As noted in the introduction, this construction is closely related to those of Entov \cite{entov:surgery} and Dimitroglou Rizell \cite{rizell:surgery}.

\subsection{Attaching Regions and Handle Attachment}

We begin by specifying the data necessary to attach a Lagrangian $q$-handle along a Legendrian submanifold $\leg$. We denote the cusps of the front of a Legendrian $\leg$ by $\leg^{\succ}$, and if $\leg$ has a generating family, then the cusps that represent births/deaths between critical points of indices $j$ and $j+1$ are denoted  $\leg^{\succ j}$.  Finally, denote a $k$-dimensional disk of radius $r$ by $D^k(r)$. 

\begin{defn} \label{defn:core}
  For $1 \leq q \leq n$, a smooth embedding $$\sigma: D^q(1+\lambda) \times D^{n-q}(\lambda) \times D^1(\lambda) \to M^n \times \rr,$$
for some small $\lambda$, is a \dfn{$q$-attaching region} for a Lagrangian $q$-handle along the front  of a Legendrian submanifold  $\leg$ if:
  \begin{enumerate}
  \item $\sigma^{-1}(\leg) = \left\{(u,v,w) \;:\; w^2 = (\| u \|^2 - \| v \|^2 -1)^3 \right\}$;
  \item $\sigma^{-1}(\leg^\succ) = \left\{(u,v,0) \;:\; \|u\|^2-\|v\|^2 = 1\right\}$, with the image of this set  called the \dfn{surgery domain}; and 
  \item For each fixed $(u_0,v_0) \in D^q(1+\lambda) \times D^{n-q}(\lambda)$, $\sigma(\{u_0\} \times \{v_0\} \times D^1(\lambda))$ is parallel to the $z$ direction in $J^1M$.
  \end{enumerate}
  
\begin{figure}
	\centerline{\includegraphics{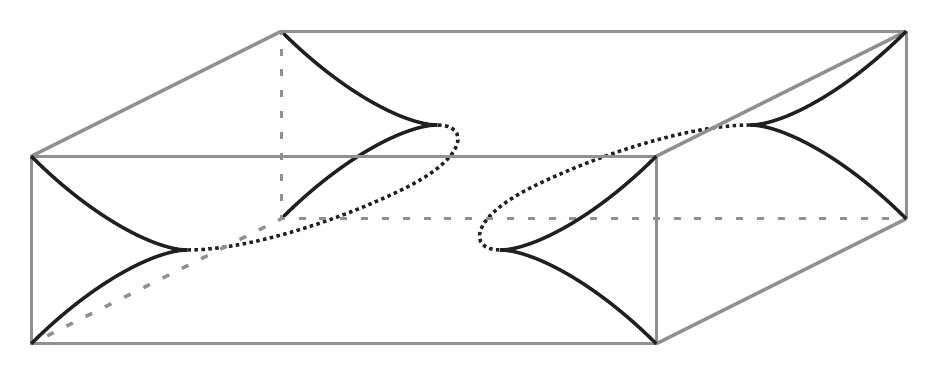}}
	\caption{A schematic picture of a the domain of an attaching region $\sigma: D^q(1+\lambda) \times D^{n-q}(\lambda) \times D^1(\lambda) \to M^n \times \rr$ with the preimage of the Legendrian $\leg$ and the cusps $\leg^\succ$ shown in solid and, respectively, dotted curves. }
	\label{fig:attaching-region}
\end{figure}

  The embedding $\sigma$ is a \dfn{gf-attaching region} if, in addition, property (2) is modified so that the surgery domain lies in  $\leg^{\succ j}$ for some fixed $j \geq 0$.  The \dfn{core disk} of the attaching region is the image of $D^q(1) \times \{0\} \times \{0\}$. 
\end{defn}

See Figure~\ref{fig:attaching-region} for schematic picture of an attaching region. After an isotopy of $\leg$, we may assume that the 
image of the attaching region lies in the vertical slice defined by $c-\lambda \leq z \leq c+\lambda$ and that the core disk lies in the hypersurface defined by $z=c$.

We are now ready to formally state the surgery construction.

\begin{figure}
  \centerline{\includegraphics[width=2.2in]{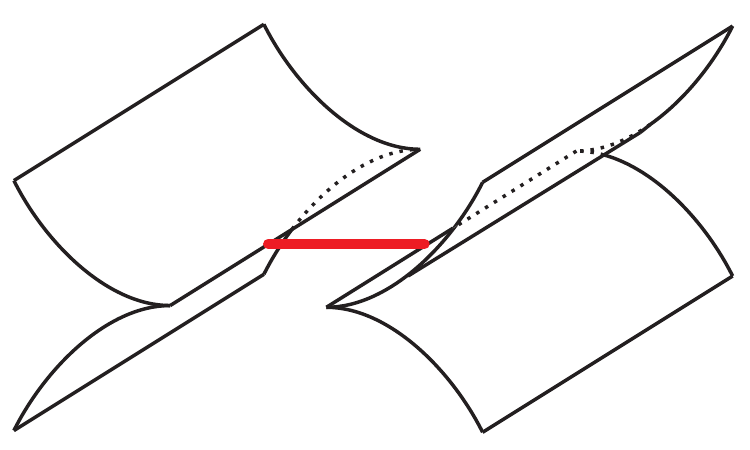} \quad \quad \includegraphics[width=2.2in]{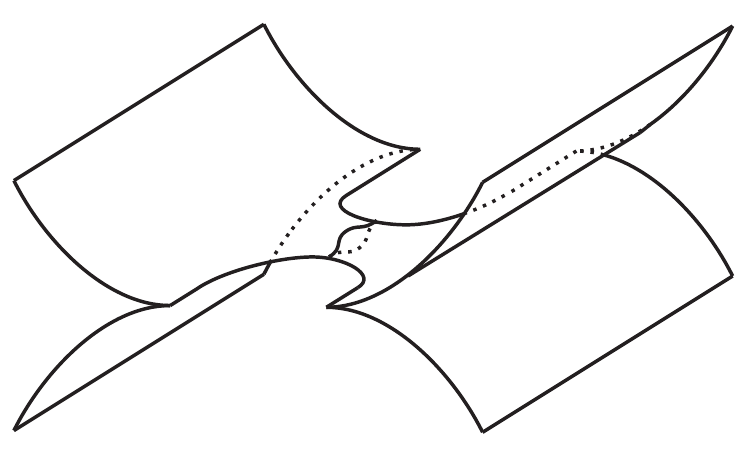}}
  \caption{(a) The core disk for a $0$ surgery and (b) the result of embedded surgery along the core disk for a Legendrian suface in $\rr^5 = J^1\rr^2$.}
  \label{fig:2d-0-surgery}
\end{figure}

\begin{thm} \label{thm:q-surgery} Let $\leg_-$ be a Legendrian submanifold of $J^1M$ with a tame  
generating family $f_-: M^n \times \rr^N \to \rr$.  Given a $q$-attaching region $\sigma$ for $\leg_-$, there exists a smooth $1$-parameter family of functions $f_t: M \times \rr^N \to \rr$ so that $F(t, x, \e) = t f_t(x, \e)$ is tame
and generates an embedded Lagrangian cobordism $(\leg_-, f_-) \prec_{(\overline{L}, F)} (\leg_+, f_+)$ satisfying:
  \begin{enumerate}
  \item the cobordism $\overline L$ has the homotopy type of a cylinder over $\leg_-$ with a $q$-cell attached, and
  \item the Legendrian $\leg_+$ is obtained from $\leg_-$ by an embedded $(q-1)$-surgery along the boundary of the core disk.
  \end{enumerate}
\end{thm}

See Figure~\ref{fig:2d-0-surgery} for an example of performing $0$-surgery on a two-dimensional Legendrian in $\rr^5$.  Note that, as a special case, the theorem allows the attachment of gf-compatible Lagrangian $0$-handle, which results is a null-cobordism for the standard $n$-dimensional flying saucer.

The remainder of this section is devoted to a proof of the surgery construction theorem. The proof begins by transferring the attaching region to the domain of the generating family $f_-$.  We then attach a handle in $\rr_+ \times M \times \rr^N$ to the fiber critical set of $f_-$.  With this scaffolding in place, we construct $F$ itself. The proof ends with a verification  that $F$ is, indeed, a generating family and  the Lagrangian cobordism that it generates is embedded. 

\subsection{\texorpdfstring{Transferring the Attaching Region to $M \times \rr^N$}{Transferring the Attaching Region to MxRn}}

Condition (3) of Definition~\ref{defn:core} guarantees that we may think of $D^q(1+\lambda) \times D^{n-q}(\lambda)$ as being embedded in $M$.  The region of the domain of $f_-$ involved in the attaching of the handle is then:
$$\tilde{E} = \left(D^q(1+\lambda) \times D^{n-q}(\lambda) \times \rr^N \right) \cap f_-^{-1}(c-\lambda, c+\lambda).$$
We let $\tilde{E}_{(u,v)}$ denote points in $\tilde{E}$ whose first two coordinates are $(u,v)$.

The structure of the fiber critical set  $\Sigma_-$ of $f_-$ in $\tilde{E}$ is quite simple to describe using conditions (1) and (2) of Definition~\ref{defn:core}.  In terms of coordinates $(u,v)$ on $D^q(1+\lambda) \times D^{n-q}(\lambda)$, we have:
\begin{enumerate}
	\item if $\|u\|^2-\|v\|^2<1$, then $\Sigma_- \cap \tilde{E}_{(u,v)} = \emptyset$;
	\item if $\|u\|^2-\|v\|^2=1$, then $\Sigma_- \cap \tilde{E}_{(u,v)}$ consists of a single point;
	\item if $\|u\|^2-\|v\|^2>1$, then $\Sigma_- \cap \tilde{E}_{(u,v)}$ consists of two points.
\end{enumerate} 
Thus, after a fiber-preserving diffeomorphism of $\tilde{E}$ that preserves each $\tilde{E}_{(u,v)}$, we may assume that $\Sigma_- \cap \tilde{E}$ is the set
	$$\{ (u,v,\eta) : \|u\|^2-\|v\|^2-\eta_1^2=1, \  \eta_2 = \cdots = \eta_N = 0\}.$$
For future use, we let $E = \tilde{E} \cap \{\eta_2 = \cdots = \eta_N = 0\}$ and we let $Q(u,v,\eta_1) = -\|u\|^2+\|v\|^2+\eta_1^2$.

\subsection{\texorpdfstring{Attaching a Handle to $\Sigma_-$}{Attaching a Handle to Sigma-}}

We next build the fiber critical set $\Sigma$ of $F$ by attaching a $q$-handle to $\Sigma_-$.  The set $\Sigma \subset \rr_+ \times M \times \rr^N$  will consist of:
\begin{enumerate}
\item a cylindrical extension $(0,t_-] \times \Sigma_{-}$;
\item a $q$-handle attached to this cylinder for $t \in [t_-, t_+]$, whose boundary is  $\Sigma_- \cup \Sigma_+ = \Sigma \cap \{t_\pm\}$; and
\item a cylindrical extension $[t_+, \infty) \times \Sigma_+$.
\end{enumerate}

\begin{figure}
		\labellist
		\pinlabel $\eta$ [t] at 80 18
		\pinlabel $u$ [l] at 206 54
		\pinlabel $v$ [b] at 108 90
		\endlabellist

  \centerline{\includegraphics[width=4in]{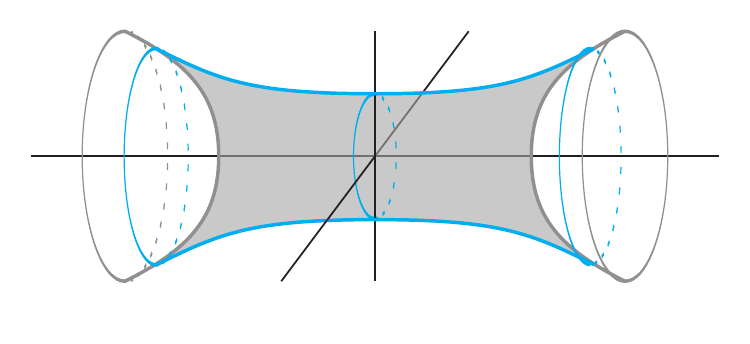}}
  \caption{The handle $H$ in the domain of $f_-$, with $\pi_x(H)$  shown in grey. }
  \label{fig:fcs-handle}
\end{figure}

To construct the $q$-handle, we first form its projection $H$ to $E \subset M \times \rr^N$.  As in the standard construction of a handle (see \cite[\S3]{milnor:morse}), we let $H$ be the deformation retract of the sublevel set $Q^{\leq 1}$ to the region diffeomorphic to $D^{q} \times D^{n-q+1}$ depicted in Figure~\ref{fig:fcs-handle}.  We use the identification of $H$ with $D^{q} \times D^{n-q+1}$ to split the boundary of $H$ into three pieces:  
\begin{enumerate}
\item $S = S^{q}(1+\lambda) \times S^{n-q+1}(\lambda)$; 
\item $\partial_- H = S^{q-1} \times D^{n-q} \setminus S$, i.e.\ the portion of the boundary of $H$ that lies in $\Sigma_-$ (not including $S$); and 
\item $\partial_+ H = D^{q} \times S^{n-q-1} \setminus S$. Note that the closure of $\partial_+ H$ is meant to be tangent (to all orders) to $\Sigma_-$.
\end{enumerate}

To place the handle $H$ into the domain of $F: \rr_+ \times M \times \rr^N \to \rr$, we consider a smooth function $h: H \setminus S \to \rr$ that satisfies:
\begin{enumerate}
\item $h$ has a single critical point at the origin of critical value $t_c \in (t_-,t_+)$;
\item near the origin, $h(u,v,\eta_1) = Q(u,v,\eta_1) + t_c$; 
\item $h^{-1}(t_\pm) = \partial_\pm H$; and
\item over $\partial_\pm H$, the graph of $h$  is tangent (to all orders) to the vertical  cylinder $\partial_\pm H \times \rr$.
\end{enumerate}

We construct $\Sigma$ by taking the union:
\begin{align*}
	\Sigma = &(0,t_-] \times \Sigma_{-} \\
		& \cup \left(\Sigma_- \setminus \partial_- H \right) \times [t_-,t_+] \\
		&\cup \text{graph}(h) \\
		& \cup [t_+, \infty) \times \Sigma_+.
\end{align*}
That $\Sigma$ is smooth follows from the last two conditions in the definition of $h$.

A key feature of the $\Sigma$ constructed above is that for all $x \in \pi_x(\Int H)$, the cardinality of the set 
$$\Sigma_{(t,x)} = \Sigma \cap h^{-1}(\{t\}) \cap \pi_x^{-1}(\{x\}) \cap f^{-1}(c-\lambda, c+\lambda)$$ is an increasing function of $t$, passing from $0$ to $1$ (for at most one value of $t$) to $2$.

The following lemma is obvious from the construction above, as we essentially use  the classical Morse-theoretic picture of \cite{milnor:morse}:

\begin{lem} \label{lem:sigma-surg}
  $\Sigma_+$ is obtained from $\Sigma_-$ by a $(q-1)$-surgery.
\end{lem}

\subsection{\texorpdfstring{Constructing $F$ Using $\Sigma_-$}{Constructing F Using Sigma-}}

The next step in the proof of Theorem~\ref{thm:q-surgery} will be to construct a $1$-parameter family $f_t: M \times \rr^N \to \rr$, with $t \in \rp$, so that $\Sigma$ is the fiber critical set of $F(t, x, \eta) = t f_t(x, \eta)$.  Constructing the family $f_t$ is equivalent to constructing a smooth family of functions $f_{(t,x)}: \rr^N \to \rr$ for $(t,x) \in \rp \times M$ so that the critical points of $f_{(t,x)}$ are precisely the points of $\Sigma_{(t,x)}$.

For $x \not \in \pi_x(H)$, simply let $f_{(t,x)} = (f_-)_x$. Note that in this case, $f_{(t,x)}$ is clearly linear-at-infinity.

Now suppose that $x \in \pi_x(H)$.  As noted above, the cardinality of $\Sigma_{(t,x)}$ increases from $0$ to at most $2$, and contains a single point for at most one value $t_0$ of $t$.  Outside of a neighborhood of $\Sigma_{(t_0,x)}$, let $f_{(t,x)} = (f_-)_x$. If there exists $t_0 \in [t_-,t_+]$ so that $\Sigma_{(t_0,x)}$ consists of a single point $(t_0,x,\eta_x)$, then the construction of $f_{(t,x)}$ proceeds as follows: working one $x$ slice at a time, we modify $(f_-)_x$ in a neighborhood of $(x,\eta_x)$ so that there are no new critical points for $t_- \leq t < t_0$, there is one birth-death critical point for $t=t_0$, and for $t_+ \geq t > t_0$ there is a pair of non-degenerate critical points of indices $j+1$ and $j$ at positions dictated by $\Sigma_{(t,x)}$.  The local nature of this modification shows that $f_{(t,x)}$ has exactly the same behavior at infinity as $f_-$. This finishes the construction of $f_{(t,x)}$, and hence the construction of the generating family $F$.
   
The proof of Theorem~\ref{thm:q-surgery} will be completed by the proofs of the following two claims:

\begin{claim} 
  After a small perturbation, $F$ is a generating family for a Lagrangian cobordism between $\leg_-$ and $\leg_+$.
\end{claim}

The claim essentially follows since we constructed $\Sigma = (\partial_\eta F)^{-1}(0)$ to be a submanifold of $\rr_+ \times M \times \rr^N$; 
if necessary, a slight perturbation of $F$ will guarantee that $\mathbf{0}$ is a regular value of $\pd{F}{\e}$.

\begin{claim} \label{claim:embedded}
  The Lagrangian $L$ generated by $F$ is embedded.
\end{claim}

To prove this claim, let $\delta_{t}(x, \e, \te)$ be the difference function of $f_t$.  From Lemma~\ref{lem:emb-family}, we know the Lagrangian generated by $F$ will be embedded as long as for all $(x, \e, \te)$ with $\e \neq \te$ in the fiber-critical set of $\delta_{t}$, we have
  $$
\partial_x \delta_t(x, \e, \te) = 0 \implies \delta_t(x, \e, \te) \neq -t \partial_t \delta_t(x, \e, \te).
$$  

First, consider the case where $\left( \pi_x^{-1}(\pi_x (H) \right) \cap \Sigma_{-} \subset H$; that is, in the front projection, there are no additional portions of $\pi_{xz}(\leg_-)$ above or below the image of the core disk.  In this case, the only potential immersion points arise from pairs of points inside $H$.  By modifying $f_t$, we can guarantee that the only potential immersion points arise from a pair at the ``center'' of the handle: namely, for $(x, \e, \te)$ with $\e \neq \te$ in the fiber-critical set of $\delta_t$,
$$\partial_x \delta_t(x, \e, \te) = 0 \iff t >t_c \text{ and } x=0.$$
If $\e$ and $\te$ are labeled so that $\delta_t(x, \e, \te) > 0$, then by modifying the movement of the critical values of $f_t$, we can guarantee that $\partial_t \delta_t(x, \e, \te) > 0$, and embeddedness of $L$ follows.

Second, consider the case where potential immersion points arise from pairs of points $(t, x, \e)$ and $(t, x, \te)$ where one point is inside $H$ and the other point lies in the portion of $\Sigma$ that is cylindrical over $\Sigma_{-}$.  In this case, by perturbing $f_t$, we can ensure that there is a compact set of potential immersion points where $\partial_x \delta_t(x, \e, \te) = 0$, and that $|\delta_t(x, \e, \te)| \geq h > 0$ on the domain of interest.  Then, using an argument similar to the proof of Proposition~\ref{prop:isotopy},  
by  increasing $t_+$ and reparameterizing $f_t$, we can guarantee that
$ |\delta_t(x, \e, \te)|/t  \geq h/t > | \partial_t \delta_t(x, \e, \te)|$ 
at all the potential immersion points.  This completes the proof of Claim ~\ref{claim:embedded}, and hence the proof of Theorem~\ref{thm:q-surgery}.

 A brief examination of the proof shows that the hypothesis of a global generating family for the Legendrian $\leg_-$ is unnecessary so long as we do not expect either the Lagrangian cobordism $\overline{L}$ or the Legendrian $\leg_+$ to have a generating family.  All we need is a generating family for $\leg_-$ in a neighborhood of the attaching sphere of the core disk, and such a family is easily constructed in a local model.  Thus, we have the following corollary:

\begin{cor} \label{cor:q-surgery}
  Let $\leg_-$ be a Legendrian submanifold of $J^1M$.  Given a $q$-attaching region for $\leg_-$, there exists an embedded exact Lagrangian cobordism $\leg_- \prec_{\overline L} \leg_+$ so that:
  \begin{enumerate}
  \item The cobordism $\overline L$ has the homotopy type of a cylinder over $\leg_-$ with a $q$-cell attached, and
  \item The Legendrian $\leg_+$ is obtained from $\leg_-$ by an embedded $(q-1)$-surgery along the attaching sphere of the core disk.
  \end{enumerate}
\end{cor}

\begin{rem}
Note that the surgery construction used in Corollary~\ref{cor:q-surgery} is equivalent to the
ambient surgery construction by Dimitroglou Rizell~\cite{rizell:surgery}. There, the Lagrangian
handle is defined (Section 4.2.2) in the front projection as the union of two graphs 
instead of our generating family description. 
\end{rem}

\section{\texorpdfstring{Constructions in Dimension $3$}{Constructions in Dimension 3}}
\label{sec:3d}
 
In the special case of Legendrian links in $\rr^3$ or $J^1S^1$, the handle attachment construction in the previous section reduces to the the existence of a Lagrangian cobordism between Legendrians whose front diagrams are depicted in Figure~\ref{fig:3d-surgery}.  These cobordisms are related to the work of Ekholm, Honda, and K\'alm\'an \cite{ehk:leg-knot-lagr-cob}, but with the added benefit that if the Legendrian knot at the bottom of the cobordism has a linear-at-infinity generating family and the cusps have corresponding indices, then the link at the top also has a generating family and the cobordism has a compatible generating family.  

In this section, we illustrate the possibilities of the construction in two families of examples. As noted in the introduction, deeper applications of these constructions appear in \cite{bty,polyfillability, positivity, ss:pi-k}.

While the description of these examples follows the ``bottom up'' construction specified by the handle attachment of the previous section, we note that these examples were discovered by working ``top down'' from the resulting links, using graded normal rulings to keep track of which pairs of critical points could be canceled (see \cite{chv-pushkar, fuchs:augmentations}, for example).

\begin{figure}
  \centerline{\includegraphics[height=1.5in]{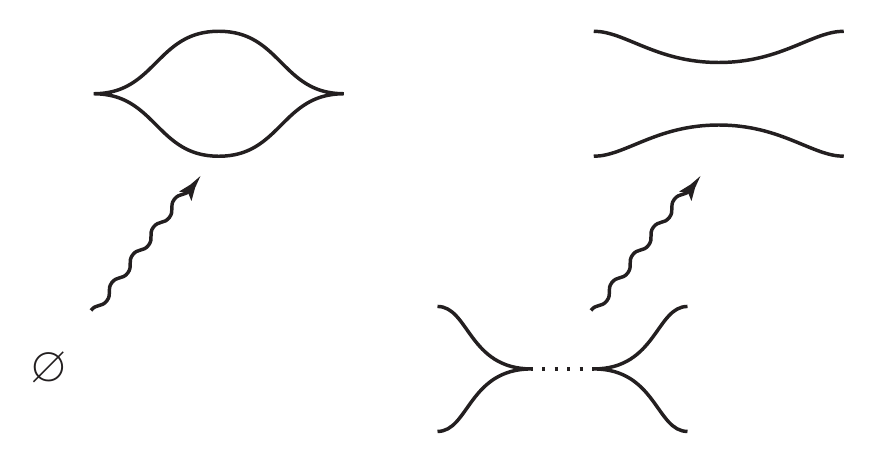}}
  \caption{The modifications to the front of a 1-dimensional Legendrian that arise from attaching a $0$-handle and a $1$-handle.}
  \label{fig:3d-surgery}
\end{figure}

\subsection{Example: Whitehead Doubles}
\label{ssec:whitehead}

Recall that for any Legendrian knot $\leg \subset \rr^3$, one can form the Legendrian Whitehead double, denoted $Wh_{tb}(\leg)$, as follows.  First form
the $2$-copy of $\leg$, which is a link consisting of $\leg$ and a small push-off of $\leg$ in the $z$ direction; then make this Legendrian link into a Legendrian  knot by replacing a $0$-tangle with a cusped $\infty$-tangle.  The top of Figure~\ref{fig:wh-double} illustrates $Wh_{tb}(\leg)$ where
$\leg$ is the Legendrian unknot with $tb = -3$ and $r = 0$.  Topologically, $Wh_{tb}(\leg)$ is the $tb(\leg)$-twisted, positively-clasped Whitehead double of the 
underlying knot type $K$ of $\leg$.

It is not difficult to show that no matter the original knot $\leg$, its Legendrian Whitehead double has $tb(Wh_{tb}(\leg) = 1$ and $r(Wh_{tb}(\leg) = 0$.  Further, it has at least one graded normal ruling, and hence a generating family by \cite{f-r}.  Moreover, we can use the techniques of the previous two sections to prove:

\begin{prop} \label{prop:wh}
	If $\leg$ be a Legendrian knot in $\rr^3$, then $Wh_{tb}(\leg)$ has a Lagrangian filling of genus $1$.  Further, if $r(\leg)=0$, then the filling is gf-compatible.
\end{prop}

\begin{proof}
We will construct a Lagrangian filling
of $Wh_{tb}(\leg)$ as follows:  first, attach a $0$-handle to obtain an unknot.  Using the trace of $\leg$, use Reidemeister type I moves at the cusps and type II moves at the crossings to drag one cusp of the unknot along $\leg$ until it lies next to the other cusp of the original unknot; see the left side of Figure~\ref{fig:wh-double} for an illustration when $\leg$ is a Legendrian unknot with no crossings.  Perform two more Reidemeister type I moves, one each on the top and bottom strands of the original unknot.  Finally, attach two $1$-handles as indicated in the center of Figure~\ref{fig:wh-double} to
obtain $Wh_{tb}(\leg)$; if $r(\leg) = 0$, then the indices of the outer cusps match.  The composition of the $0$-handle attachment, the isotopy, and the two $1$-handle attachments yields a genus $1$ Lagrangian filling for $Wh_{tb}(\leg)$ which is gf-compatible if $r(\leg) = 0$.
\end{proof}

\begin{figure}
  \centerline{\includegraphics[width=5in]{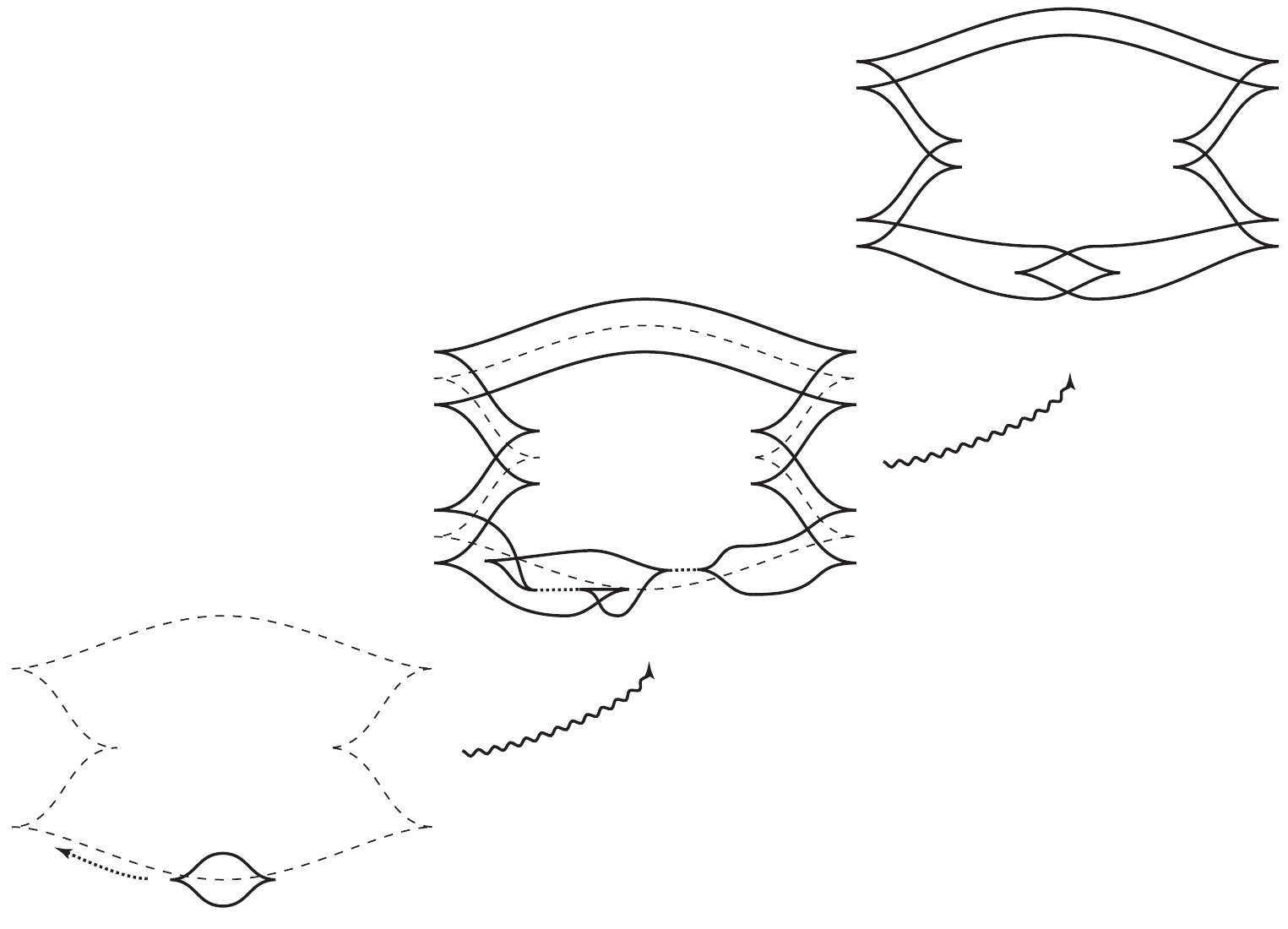}}
  \caption{The construction of a Lagrangian null-cobordism for a Whitehead double.}
  \label{fig:wh-double}
\end{figure}

By combining Proposition~\ref{prop:wh} with Theorem 1.5 of \cite{josh-lisa:obstr}, we obtain:

\begin{cor} 
  The set linearized contact homology or generating family polynomials for the Whitehead double of any Legendrian knot with rotation number $0$ must contain the polynomial $2 + t$.  
\end{cor}

Denote by $\overline{tb}(K)$ the  topological invariant
given by the maximal Thurston-Bennequin invariant of any Legendrian 
representative of a smooth knot $K$. By adding $s$ stabilizations to a maximal $tb$ representative of $K$, we obtain a Legendrian representative $\leg_s$ of $K$ with $tb(\leg_s) = \overline{tb}(K)-s$.  Applying the construction above, for all $s \geq 0$ we find a genus $1$ Lagrangian null-cobordism for $Wh_{tb}(\leg_s)$, which
is topologically the $(\overline{tb}(K)-s)$-twisted Whitehead double of the underlying knot type $K$ of $\leg$. Combining this construction with Chantraine's result that a Lagrangian filling realizes the smooth $4$-ball genus of a Legendrian knot \cite{chantraine}, we obtain a result of Rudolph:

\begin{cor}[Rudolph \cite{rudolph}]
  For any $r \leq \overline{tb} (K)$, the $r$-twisted Whitehead double of $K$ has  smooth $4$-ball genus equal to $1$.  
\end{cor}

\subsection{Example: Positive Braids}
\label{ssec:braid}

Given a positive braid $B = \sigma_{i_1} \cdots \sigma_{i_k}$ on $s$ strands, form a Legendrian link $\leg_B$ as in the top of Figure~\ref{fig:pos-braid-constr}.  

\begin{prop} \label{prop:braid}
	For any positive braid $B$ with $k$ crossings on $s$ strands, the Legendrian link $\leg_B$ with $c$ components has a gf-filling of genus $\frac{1}{2}(2-c+k-s)$.
\end{prop}

\begin{proof}
We  construct a gf-filling  of $\leg_B$ as follows:  first, for each generator $\sigma_{i_j}$, attach $(s-1)$ $0$-handles and perform an isotopy to obtain a nested unlink $U_j$ of $(s-1)$ components
  with a single crossing mimicking the $\sigma_{ij}$ crossing in $\leg$;
for an example, see the left side of Figure~\ref{fig:pos-braid-constr}.  Position these unlinks next to each other, ordered from $U_1$ to $U_k$.  Finally, successively attach $s$ $1$-handles
 between each pair of adjacent unlinks, successively starting at the
outermost cusps,  as in  Figure~\ref{fig:pos-braid-constr} to obtain $\leg_B$.    If $\leg_B$ has $c$ components, then this  filling has genus $\frac12(2-c+k- s)$.  
\end{proof}

In particular, when $\leg_B$ is a knot,
we see that the smooth $4$-ball genus of $\leg_B$ is $\frac12(1+k-s)$, as originally proved by  Rudolph \cite[\S3]{rudolph:qp-obstruction} as a corollary of Kronheimer and Mrowka's work on embedded surfaces in $4$-manifolds \cite{km:surfaces1}.

\begin{figure}
  \centerline{\includegraphics{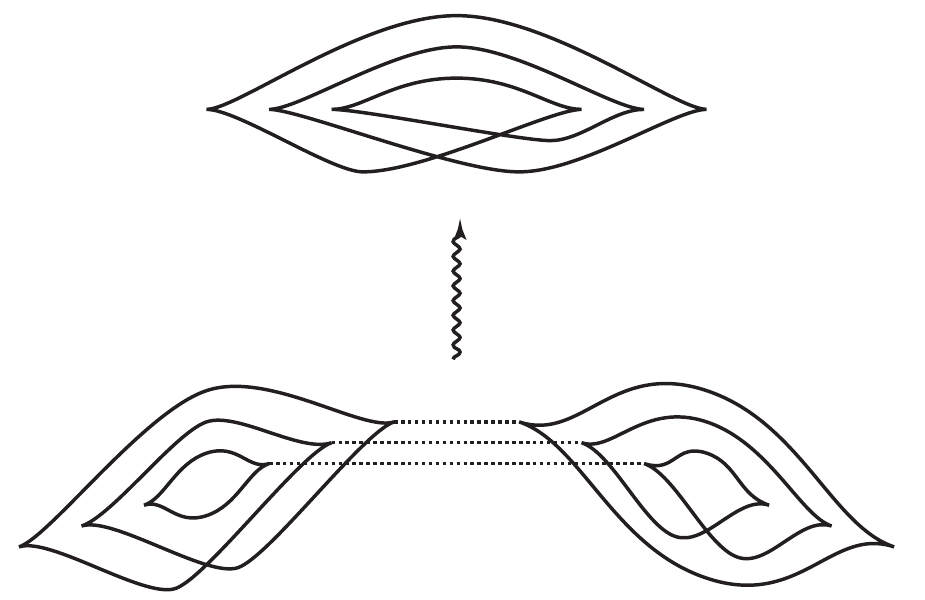}}
  \caption{The construction of a Lagrangian null-cobordism of genus $0$ for the zero closure of the positive braid $\sigma_2 \sigma_1$.  The $1$-handles are attached in succession from the outermost to innermost cusps.  }
  \label{fig:pos-braid-constr}
\end{figure}

\begin{rem}  In a similar spirit, in \cite{bty} Lagrangian fillings of positive, Legendrian rational links are constructed.  In particular, it is shown that the smooth $4$-ball genus of these positive 
rational knots  can be calculated from its rational notation, see \cite[Remark  1.8]{bty}.
\end{rem}

\begin{rem}
The discussion in this subsection  is the beginning of a more interesting story about the relationship between various notions of positivity (braid positivity, positivity, and (strong) quasi-positivity) and the existence of a Lagrangian filling; see \cite{positivity} for a deeper exploration.
\end{rem}

\section{Legendrian Geography}
\label{sec:geo-bot}
 With the isotopy, spinning, and, most importantly, the handle attachment constructions in hand, we proceed to apply them to two questions about the geography of Legendrian submanifolds:  a non-classical geography question (which generating family polynomials can be realized by Legendrian submanifolds?) and the classical fillable geography question (what Thurston-Bennequin numbers can be realized by fillable Legendrian submanifolds?).  Throughout this section, we work with coefficients in a field \ff.

\subsection{Duality and Compatible Polynomials}
\label{ssec:duality}

After strengthening the duality exact sequence for generating family cohomology of \cite{josh-lisa:obstr} to better take into account the algebraic topology of the underlying Legendrian, we prove Theorem~\ref{thm:compatible}, which restricts the possible Poincar\'e polynomials for the generating family cohomology.  Analogous versions of these results for linearized contact homology appear in
\cite{high-d-duality}. The  strengthening of the duality exact sequence takes two forms.  First, we relate the maps in the duality exact sequence to the Poincar\'e duality of the Legendrian; second, we prove the existence of a ``fundamental class'' for the generating family homology. 

\begin{thm}[Duality] \label{thm:duality} If $\leg$ is a Legendrian submanifold of $J^1M$ with linear-at-infinity generating family $f$, then there is a long exact sequence:
  \begin{equation} \label{eqn:duality-les}
    \xymatrix@R=7pt@C=15pt{
     \cdots \ar[r] & GH^{k-1}(f) \ar[r]^-{\rho_k} & GH_{n-k}(f)
      \ar[r]^-{\sigma_k} & H^{k}(\Lambda) \ar[r]^-{\delta_k}
      & \cdots.   }
  \end{equation}
	The maps $\delta_k$ satisfy two further properties:
	\begin{enumerate}
	\item If $\gamma: H^k(\leg) \to H_{n-k}(\leg)$ is the Poincar\'e duality isomorphism, then, when using coefficients in a field, $\gamma \circ \sigma_k$ is the adjoint of the map $\delta_{n-k}$.
	\item The map $\delta_n$ does not vanish. In particular, over a field, it is an isomorphism when $\leg$ is connected.
	\end{enumerate}
\end{thm}

We will delay the proof of this theorem until the appendix, as the somewhat technical proof uses methods that are quite different than those in the rest of the paper. We  call an element $\alpha \in \im \delta_k$ a \dfn{manifold class}.  If $\leg$ is connected, then the image under $\delta_n$ of the top class of $\leg$ in $GH^n(f)$ is called the \dfn{fundamental class}.

\begin{rem}
	If we were to use the total generating family cohomology instead of the relative version, then part (2) of the duality theorem above and long exact sequence (\ref{eq:gh-les}) would imply that there is no fundamental class in $\widetilde{GH}\hphantom{}^n(f)$, but that there is a manifold class in degree $0$.  This indicates that the total generating family cohomology would be more convenient for the study of the cohomology \emph{ring}, as the degree $0$ manifold class would constitute a unit.  
\end{rem}

The strengthened duality theorem is the key ingredient in the proof of Theorem~\ref{thm:compatible}, which states that every generating family polynomial is compatible with duality.

\begin{proof}[Proof of Theorem~\ref{thm:compatible}]
To set notation, let $d_k = \dim GH^k(f)$, let $q_k = \dim \im \delta_k$, let $b_k = \dim H^k(\leg)$, and let $p_k = d_k - q_k$.

We begin by proving that, under the hypotheses of Theorem~\ref{thm:duality}, when using field coefficients we have
\begin{equation} \label{eq:betti} 		
	b_k = q_k+q_{n-k}.
\end{equation}

First note that since $\delta_{n-k}$ and $\sigma_k$ are adjoints up to the isomorphism $\gamma$, the dimensions of their kernels are the same.  The relation (\ref{eq:betti}) now follows from the rank-nullity theorem and the exactness of the sequence (\ref{eqn:duality-les}) at $H^k(\leg)$.

We next claim that for all $k \in \zz$, we have:
\begin{equation} \label{eq:pk}
p_k = p_{n-1-k}
\end{equation}
The rank-nullity theorem for $\rho_k$ and the exactness of the sequence (\ref{eqn:duality-les}) at $GH^k(f)$ imply that 
\begin{equation} \label{eq:dk1}
	d_k = q_k + \dim \ker \sigma_{k+1}.
\end{equation}
Equation (\ref{eq:betti}) and the rank-nullity theorem for $\delta_k$ then imply that $q_{n-1-k} = b_{k+1} - q_{k+1} = \dim \ker \delta_{k+1}$.  Combining this fact with the rank-nullity theorem for $\sigma_{k+1}$ and the exactness of the sequence (\ref{eqn:duality-les}) at $GH_{n-1-k}(f)$, we obtain:
\begin{equation} \label{eq:dk2}
	d_{n-k-1} = \dim \ker \sigma_{k+1}  + q_{n-1-k}.
\end{equation}
Combining Equations~(\ref{eq:dk1}) and (\ref{eq:dk2}) then yields
\begin{equation} \label{eq:dk3}
	d_k = q_k + p_{n-1-k},
\end{equation}
and Equation~(\ref{eq:pk}) now follows from the definition of $p_k$.

The theorem is now a consequence of Equations~(\ref{eq:betti}), (\ref{eq:pk}), and (\ref{eq:dk3}) with $q_k$ and $p_k$ forming the coefficients of $q(t)$ and $p(t)$, respectively.  The fact that $q_n \neq 0$ is a consequence of Theorem~\ref{thm:duality}.
\end{proof}

\subsection{Non-Classical Geography}
\label{ssec:non-classical-geo}

In the last section, we proved that every generating family polynomial of a connected Legendrian submanifold of $J^1M$ is compatible with duality; in this section, we use the constructions of Sections~\ref{sec:basic} and \ref{sec:surgery} to prove Theorem~\ref{thm:geography}, namely that every Laurent polynomial in connected form that is compatible with duality is the generating family polynomial for \emph{some} connected Legendrian submanifold of dimension $n \geq 2$.    We will construct the Legendrian in a  $J^1\rr^n$ coordinate chart inside $J^1M$. 

Before beginning the constructions proper, we set down a useful computation for the generating family cohomology of a gf-compatible $0$-surgery connecting two connected Legendrians.

\begin{lem}[$0$-surgery Lemma] \label{lem:0-surg}  Suppose $\leg_-$ has a tame generating family $f_-$  so 
that $\Gamma_{f_-}(t)$ is of the form
$$\Gamma_{f_-}(t) = \left(2t^n + q_{n-1} t^{n-1} + \dots + q_1t  + q_0\right) + p(t) + t^{n-1}p(t^{-1}),$$
where $p(t) = \sum_{i \in \mathbb Z, i \geq \lfloor \frac{n-1}{2} \rfloor} p_i t^i$.  
If $\leg_+$ is a connected Legendrian obtained from $\leg_-$ by a gf-compatible $0$-surgery, then $\leg_+$ has a tame
generating family $f_+$ with 
\[\Gamma_{f_+}(t) = \Gamma_{f_-}(t) - t^n.\]
\end{lem}

\begin{proof}  Let $L$ denote the $(n+1)$-dimensional Lagrangian cobordism between the Legendrians $\leg_-$ and $\leg_+$ described in the hypotheses of the lemma.
Since $L$ is obtained from $\Lambda_-$ by attaching a $1$-handle, $L$ is obtained from
$\Lambda_+$ by attaching an $n$-handle.   
When $k$ is neither $n$ nor $n-1$, the Cobordism Exact Sequence of Theorem \ref{thm:cobord-les} implies that $GH^k(f_-) \simeq GH^k(f_+)$.  The remaining terms of the Cobordism Exact Sequence are:
\begin{equation} \label{eqn:0-surg-les}
0 \to GH^{n-1}(f_-) \to GH^{n-1}(f_+) \to \mathbb F \to GH^{n}(f_-) \to GH^{n}(f_+) \to 0 .
\end{equation}
Since $\dim GH^n(f_-) = 2+p_n$ by hypothesis, the exactness of the sequence above implies that either $\dim GH^n(f_+) = 1+p_n$ or $\dim GH^n(f_+) = 2+p_n$.  To see which of these is correct, we turn to the Duality Exact Sequence (\ref{eqn:duality-les}), which tells us that:
\begin{equation*}
 0 \to GH^{-1}(f_+) \to GH_{n}(f_+) \to H^0(\leg_+) \to \cdots.
\end{equation*}
We already know the isomorphism $GH^{-1}(f_+) \simeq GH^{-1}(f_-)$, and thus we have that
$\dim GH^{-1}(f_+) = p_{-1} = p_n$.  Since $\leg_+$ is connected, we use Theorem~\ref{thm:duality}(2) to 
conclude that $\dim GH_{n}(f_+) = \dim GH^{n}(f_+)$ must be $1 + p_n$.  It then follows from the exactness of the sequence (\ref{eqn:0-surg-les}) that $\dim GH^{n-1}(f_+) = q_{n-1}+p_{n-1}$, thus proving the lemma.
\end{proof}

One situation in which we can use this lemma is in the connected sum of two connected Legendrians in $J^1\rr^n$. Given two Legendrians with tame generating families $(\leg_i,f_i)$, $i=1,2$, sufficiently separated in the horizontal direction, there is a tame generating family $f_-$ for their disjoint union with
\[\Gamma_{f_-}(t) = \Gamma_{f_1}(t) + \Gamma_{f_2}(t);\]
see Proposition 3.19 of \cite{josh-lisa:obstr} for more details. Let $(\leg_+,f_+)$ be the result of a gf-compatible $0$-surgery on $(\leg_-,f_-)$ that connects the two components of $\leg_-$.  

\begin{cor}[Connected Sums] \label{cor:connect-sum}
	The generating family cohomology of the connect sum $(\leg_+,f_+)$  may be computed as follows:
	\[ \Gamma_{f_+}(t) = \Gamma_{f_1}(t) + \Gamma_{f_2}(t) - t^n.\]
\end{cor}

The first important step in the proof of Theorem~\ref{thm:geography} is to build up Legendrians with interesting manifold classes, which are represented by the polynomial $q_0 + q_1t + \cdots + q_nt^n$ in Equation~(\ref{eqn:poly}).

\begin{lem}[Manifold Lemma] \label{lem:mfld} 
For all integers $a = 1, \ldots, n-1$,
there exists a connected, $n$-dimensional Legendrian submanifold $\leg_a \subset J^1M$ with a tame generating family $f_a$ such that 
\[\Gamma_{f_a}(t) = t^n + t^a.\]
In addition, there exists a $2$-component $n$-dimensional Legendrian link $\leg_0$ with a tame generating family $f_0$ such that
\[\Gamma_{f_0}(t) = t^n + t^0.\]
\end{lem}

\begin{figure}  
		\labellist
		\small
		\pinlabel $D^{n}$ [lt] at 155 52
		\endlabellist
  \centerline{\includegraphics{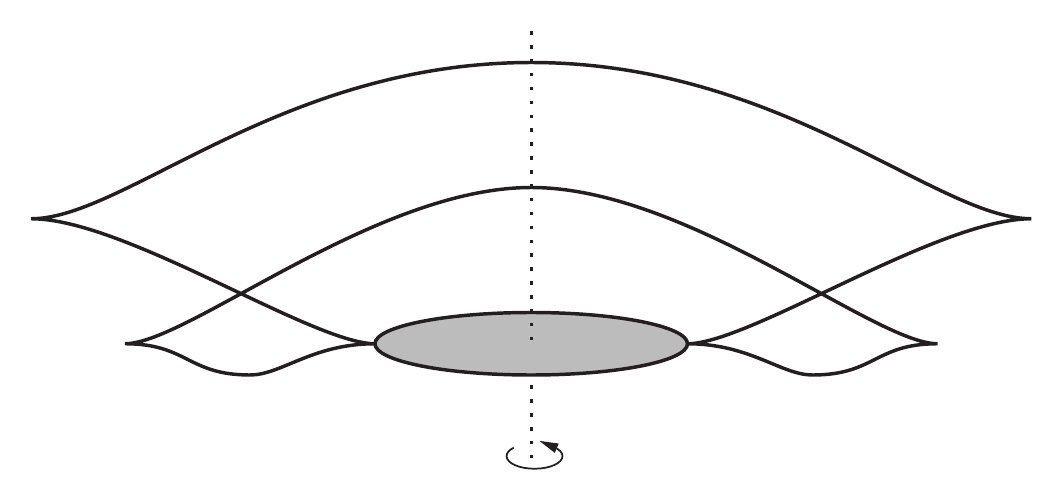}}
  \caption{Spinning this unknot results in an $n$-dimensional Legendrian submanifold isotopic to the standard flying saucer, but having a $S^{n-1}$ family of cusps that bounds a disk $D^n$ on the outside.}
  \label{fig:mfld-lemma}
  \end{figure}

\begin{proof} In this proof and in the proofs below, we will construct the Legendrian $\leg_a$ in a  $J^1\rr^n$ coordinate chart inside $J^1M$. 
Let $\leg_-$ denote the $n$-dimensional Legendrian sphere obtained from spinning the front of the Legendrian
unknot shown
in Figure~\ref{fig:mfld-lemma}.  Since Legendrian isotopy commutes with spinning, we know that $\leg_-$ is isotopic to the standard $n$-dimensional flying saucer.  
{This implies that $\leg_-$ is gf-fillable (by attaching a $0$-handle as in Theorem~\ref{thm:q-surgery} and by Proposition~\ref{prop:isotopy}).}
  By construction, $\leg_-$ contains an $(n-1)$-dimensional
sphere of cusps that bounds a horizontal disk $D^{n}$.  For $a = 0, 1, \dots, n-1$, perform an $(n-a-1)$-surgery on $\leg_-$, which, when combined with the Lagrangian filling of $\leg_-$ also yields a gf-filling $(L_a,F_a)$ of $(\leg_a, f_a)$.   For $a > 0$, the resulting Legendrian $\leg_a$ is connected. 

We now proceed to compute the generating family cohomology of the links $\leg_a$.  Since the Lagrangian $L_a$ was constructed by attaching an $(n-a)$-handle to a $0$-handle, its homology is supported in dimensions $0$ and $n-a$, both with dimension $1$.  Poincar\'e duality then implies that 
\[ H^k(L_a,\leg_a) = \begin{cases} \ff & k=a+1,n+1, \\ 0 & \text{otherwise.} \end{cases}\]
The Cobordism Exact Sequence implies that $GH^k(f_a) \simeq H^{k+1}(L_a,\leg_a)$, and the lemma follows.
\end{proof}

\begin{rem}
	The proof of Lemma~\ref{lem:mfld} shows that the Legendrians $\leg_a$ are all gf-fillable.
\end{rem}

Taking connect sums of the Legendrians constructed in Lemma~\ref{lem:mfld} and applying Corollary~\ref{cor:connect-sum} allows us to build up Legendrians with (almost) arbitrary manifold classes.

\begin{cor}[Manifold Class Building Block]\label{cor:mfld-block} For all nonnegative integers $q_1, q_2, \dots, q_{n-1}$,
there exists a connected,  $n$-dimensional Legendrian submanifold that has a tame generating family $f$ with 
\[\Gamma_f(t) = t^n + q_{n-1}t^{n-1} + \dots + q_1 t^1.\]
\end{cor}

We are now ready to build up the ``duality classes'' recorded by the polynomials $p(t)$ and $t^{n-1}p(t^{-1})$.  We begin by constructing Hopf links with the desired classes.

\begin{lem}[Hopf Link Lemma] \label{lem:hopf} For $n \geq 2$ and $a \neq 0, n-1$,
there exists an $n$-dimensional Hopf link $\Lambda_a$ with
a tame generating family $f$ so that
$$\Gamma_f(t) = 2t^n + t^a + t^{n-1}t^{-a}.$$
\end{lem}

\begin{figure}
		\labellist
		\small
		\pinlabel $a$ [lt] at 91 95
		\pinlabel $a+1$ [lt] at 91 112
		\pinlabel $0$ [rt] at 85 36
		\pinlabel $-1$ [rt] at 85 19
		\pinlabel $a$ [rt] at 158 157
		\pinlabel $a+1$ [rt] at 158 174
		\pinlabel $0$ [rt] at 177 36
		\pinlabel $-1$ [rt] at 177 19

		\endlabellist
  \centerline{\includegraphics[width=4in]{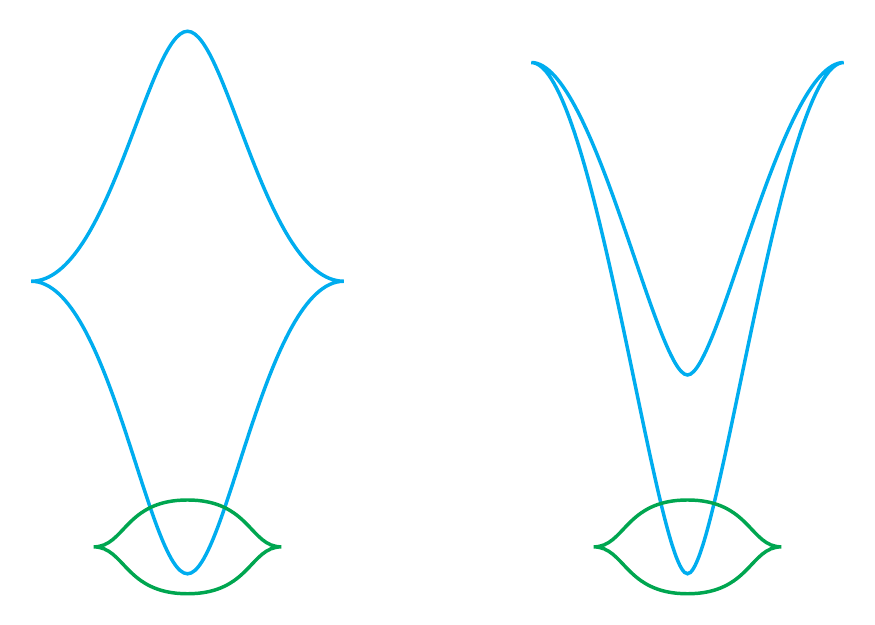}}
  \caption{Two Legendrian Hopf links used in the proof of Lemma~\ref{lem:hopf}.  Both have six Reeb chords along the central axis, with the length of the Reeb chord from the bottom of the upper component to the top of the lower component shorter than then Reeb chord between the bottoms of the two components.}
  \label{fig:hopf-lemma}
  \end{figure}
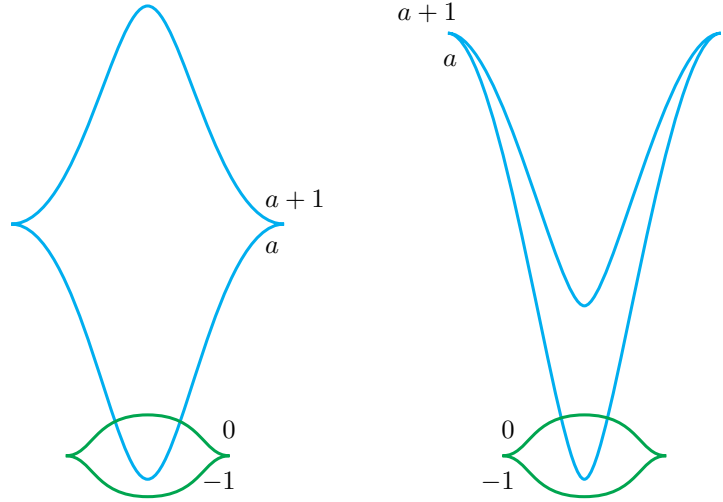

\begin{proof}  By the symmetry of these polynomials, it suffices to show it is possible
to realize the polynomial $2t^n + t^a + t^{n-1-a}$, for $a \geq \lfloor \frac{n-1}{2}  \rfloor$ and $a \neq n-1$.

We begin by considering the $1$-dimensional Legendrian Hopf link on the left side of Figure~\ref{fig:hopf-lemma}.  We claim that this link has a generating family with the indicated (relative) fiber indices.  Start with two linear-at-infinity generating families $f_0$ and $f_1$ for the lower and upper unknots, respectively, on the left side of Figure~\ref{fig:hopf-lemma}; suppose that they have the same fiber dimensions and that their fiber critical sets have the same fiber indices. Stabilize $f_1$ with a non-degenerate quadratic function of index $a+1$ and stabilize $f_0$ with a non-degenerate quadratic function of the same dimension, but of index $0$.  Finally, translate the domain of $f_1$ in the fiber so that its support is disjoint from that of $f_0$ and let $f_L$ be the sum of $f_0$ and $f_1$ in the sense of \cite[Definition 3.18]{josh-lisa:obstr}; the result $f_L$ is a generating family for the full Hopf link with the desired indices on the fiber critical sets.

Spinning the two  fronts in Figure~\ref{fig:hopf-lemma} yields isotopic $n$-dimensional Legendrian Hopf links $\leg_L$ and $\leg_R$; note that $\leg_L$ has a generating family  (again called $f_L$) by Proposition~\ref{prop:front-spin}.

Since $\leg_L$ and $\leg_R$ are isotopic and $\leg_L$ has a generating family, the persistence of generating families under isotopy implies that $\leg_R$ has a generating family $f_R$ that has the same generating family homology as $f_L$.  The difference functions for both $f_L$ and $f_R$ have six critical points with positive critical values.  The indices of these critical points, listed in order of decreasing critical value, are:
\begin{center}
\begin{tabular}{r||c|c|c|c|c|c}
	$\leg_L$ & $n+1+a$ & $n$ & $n+a$ & $n$ & $n-1-a$ & $a$ \\ \hline
	$\leg_R$ & $a+1$ & $n$ & $a$ & $n$ & $n-1-a$ & $a$
\end{tabular}
\end{center}

It follows that the critical points of $f_L$ with
indices $n+1+a$ and $n+a$ cannot contribute to the homology; similarly, for $f_R$, the critical points with
indices $a+1$ and $a$ cannot contribute to the homology,  and thus the total homology is at most $4$-dimensional.

Except in the cases  $a = n-a, n+1, n-1$, index arguments imply that  the generating family homology must be $4$-dimensional, i.e., that
\[\Gamma_f(t) = 2 t^n + t^a + t^{n-1-a},\]
as desired.

By assumption, $a \neq n-1$, so it remains to show that the generating family homology is $4$-dimensional when $a=n-a$ and when $a = n + 1$.   

When $a = n-a$, then the critical points in degrees $a$ and $n-1-a = a-1$  survive in homology since the critical value of the critical point of index $a-1$ is larger than that of the critical point of index $a$.  Thus, this case also yields the desired $4$-dimensional homology.
  
Finally, when $a = n+1$, we can apply a duality argument to show that the total homology is $4$-dimensional.  In this case, $n-1-a = -2$, and if the homology were not $4$-dimensional,
 then we would have $\Gamma_f(t) =  t^n + t^{-2}$. This is a contradiction to the  duality long exact sequence, as the part of the sequence
\[ GH^{n+1}(f) \to GH_{-2}(f) \to H^{n+1}(\leg) \]
would become
\[ 0 \to \ff \to 0.\] 
This completes the proof of the lemma.
 \end{proof}

\begin{rem} \label{rem:hopf-filling}
	As all of the Hopf links constructed in the lemma above are Legendrian isotopic to the Hopf link $\leg_0$ of Lemma~\ref{lem:mfld}, they are all fillable, though not necessarily gf-fillable.
\end{rem}

We next use the Hopf links constructed above to produce Legendrian spheres whose generating family homologies have a given pair of dual classes.

\begin{lem}[Sphere Lemma] \label{lem:sphere} For all integers $a$, there exists a 
 Legendrian $n$-sphere $\Lambda$ with a tame generating family $f$ so that
$$\Gamma_f(t) = t^n + t^a + t^{n-1}t^{-a}.$$ 
\end{lem}

\begin{proof}
By symmetry, it suffices to restrict to the case where $a \geq \lfloor \frac{n-1}2 \rfloor$. 

\begin{figure}
	\labellist
		\small
		\pinlabel $\leg_0$ [l] at 98 103
	\endlabellist
  \centerline{\includegraphics{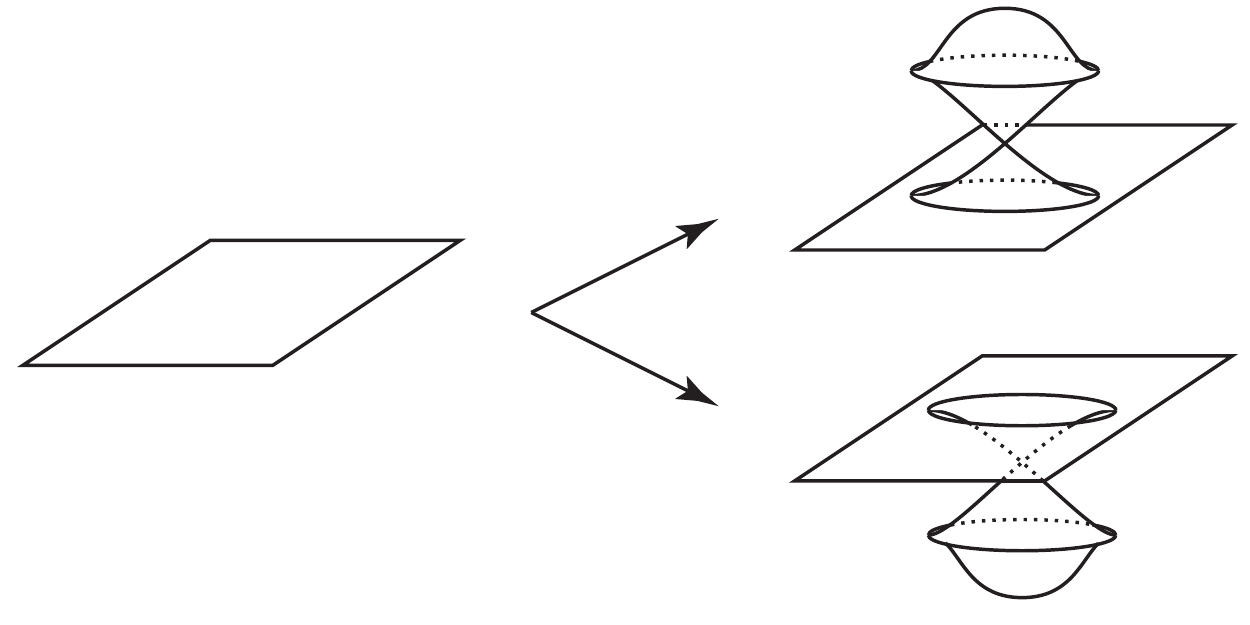}}
  \caption{The higher dimensional first Reidemeister moves.}
  \label{fig:R1move}
\end{figure}

First, we generalize the first Reidemeister move to higher dimensions as follows: for any Legendrian submanifold $\leg \subset J^1 \rr^n$, the Legendrian submanifold $\leg'$ obtained by replacing a graph-like portion of the front projection of $\leg$ with one of the fronts depicted in Figure~\ref{fig:R1move} is Legendrian isotopic to $\leg$.  To see why, let $\leg_0 \subset \leg$ be an open subset obtained as the $1$-jet of a function and
contained in a Darboux chart of $J^1\rr^n$. We claim that the 1-parameter family of generating families $f : [-1 \times 1] \times \rr^n \times
\rr^n \to \mathbb R$ defined by
$$
f_t(x,\eta) = \pm(\|\eta\|^4 + t \|\eta\|^2) + x \cdot \eta,
$$
where $x = (x_1, \ldots, x_n)$ and $\eta = (\eta_1, \ldots, \eta_n)$,
describes a Legendrian isotopy between $\leg_0$ (for $t = +1$) and 
the Legendrian submanifolds whose fronts are depicted on Figure~\ref{fig:R1move} (for $t= - 1$).
The sign $\pm$ can be chosen in order to obtained the front on the bottom or top 
of the figure.

\begin{figure}
		\labellist
		\small
		\pinlabel $a$ [r] at 37 128
		\pinlabel $a+1$ [r] at 37 148
		\pinlabel $0$ [r] at 275 45
		\pinlabel $-1$ [r] at 275 25
		\pinlabel $a+1$ [l] at 280 105
		\endlabellist
  \centerline{\includegraphics{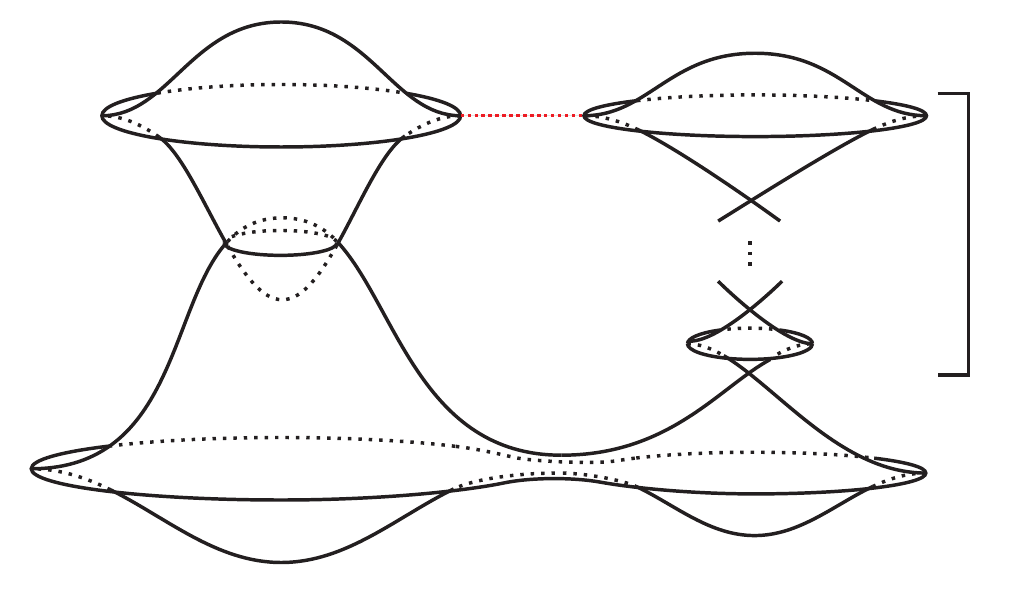}}
  \caption{A scheme to construct a Legendrian sphere with $\Gamma_{f_a}(t) = t^n + t^a + t^{n-1}t^{-a}$ with $a \geq \lfloor \frac{n-1}{2} \rfloor$.
  }
  \label{fig:sphere-lemma}
  \end{figure}
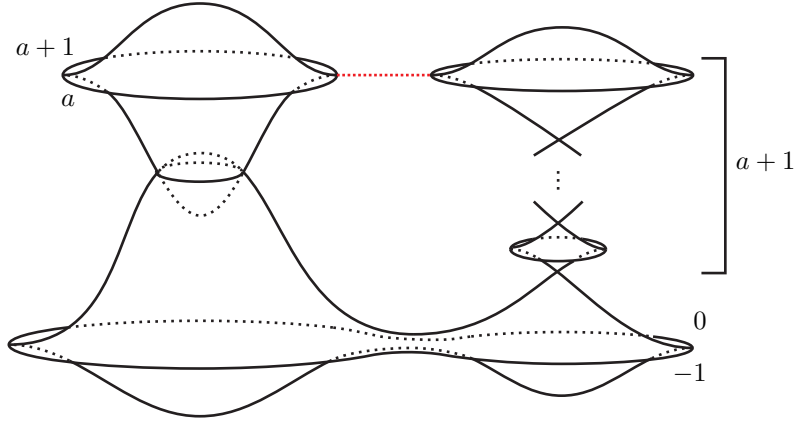

With this technique in hand, take a Hopf link constructed in Lemma~\ref{lem:hopf} and apply $a+1$  Reidemeister moves; in the case $a=0$, use the link from Lemma~\ref{lem:mfld}.  The result has a generating family $g_a$ with the fiber indices indicated in Figure~\ref{fig:sphere-lemma}. Applying a $0$-surgery along a horizontal line  indicated in Figure~\ref{fig:sphere-lemma}  produces a gf-compatible Lagrangian cobordism from the Hopf link to a sphere $\leg_a$ with generating family $f_a$. 

To finish the proof of the lemma, we verify that the generating family cohomology of $(\leg_a, f_a)$ is given by $\Gamma_{f_a} = t^n + t^a + t^{n-1}t^{-a}$. If $(a, n-1-a) \neq (n-1, 0)$, then the $0$-Surgery Lemma \ref{lem:0-surg} implies the desired result.  If, on the other hand, we are in the case $(a, n-1-a) = (n-1, 0)$, the Cobordism Exact Sequence implies that $\dim GH^0(f_0) = 1$ and that $\dim GH^k(f_a) = 0$ for all other $k \neq 0,n-1,n$.  The remaining part of the Cobordism Exact Sequence is:
\[ 0 \to GH^{n-1}(f_0) \to \ff \to \ff \to GH^n(f_0).
\]
Thus,  $GH^{n-1}(f_0)$ and $GH^n(f_0)$ either simultaneously vanish or have dimension $1$.  Theorem~\ref{thm:duality}(2) implies that the latter is, indeed, the case, which completes the proof.
\end{proof}

\begin{rem} \label{rem:sphere-filling}
	Continuing Remark~\ref{rem:hopf-filling}, we note that the spheres constructed in the lemma above are all fillable since they arise from attaching Lagrangian handles to fillable Legendrians.
\end{rem}

From this lemma, we obtain spherical building blocks for duality classes:

 \begin{cor}[Sphere Building Block]\label{cor:sphere-block} For any $n \geq 2$ and any polynomial
$p(t) = \sum_{i \in \mathbb Z, i \geq \lfloor \frac{n-1}{2} \rfloor} p_i t^i$, 
there exists an $n$-dimensional Legendrian sphere with tame generating family $f$ with polynomial
$$\Gamma_f(t) = t^n + p(t) + t^{n-1}p(t^{-1}). $$
\end{cor}

The corollary follows from Lemma~\ref{lem:sphere} in an analogous fashion to how
Corollary~\ref{cor:mfld-block} followed from Lemma~\ref{lem:mfld} and Corollary~\ref{cor:connect-sum}.

Finally, we have all the building blocks to answer the non-classical generating family geography question.

\begin{proof}[Proof of Theorem~\ref{thm:geography}]  This proof 
  follows from taking a connect sum of a 
manifold constructed in Corollary~\ref{cor:mfld-block} and a sphere constructed in Corollary~\ref{cor:sphere-block}; once again, the generating family polynomial is computed using Corollary~\ref{cor:connect-sum}.
\end{proof}

\subsection{Classical Fillable Geography}
\label{ssec:classical-fillable-geo}

The constructions built up in the previous section, especially the Sphere Building Block Corollary~\ref{cor:sphere-block}, also  suffice for the proof of Theorem~\ref{thm:classical-fillable}. We will also need the fact that the Thurston-Bennequin number of a Legendrian submanifold $\leg \subset J^1\rr^n$ with generating family $f$ may be computed by taking the Euler characteristic of the generating family homology:
\begin{equation} \label{eqn:tb}
	tb(\leg) = (-1)^{\frac{(n-2)(n-1)}{2}} \Gamma_f(-1).
\end{equation}
This fact was proven in \cite{ees:high-d-geometry} for linearized Legendrian contact homology, and can be translated to generating family homology using \cite[Proposition 3.2]{josh-lisa:obstr}.

Fix an odd $n\geq 2$ and any odd integer $2k+1$.  Suppose that the quantity $\frac{(n-2)(n-1)}{2}$ is even; the proof in the other case is entirely similar. If $2k+1 > 0$, it is straightforward to compute that the Legendrian sphere $\leg$ with generating family polynomial
\[ \Gamma_f(t) = t^n +  (k+1)t^{n-1} + (k+1)t^0\] 
has $tb(\leg) = 2k+1$. If $2k+1 < 0$, use the Legendrian sphere with generating family polynomial
\[ \Gamma_g(t) = (k+1)t^n + kt^{-1}.\]
By Remark~\ref{rem:sphere-filling}, all of the spheres used above are fillable, though only the first set is gf-fillable. This completes the proof of Theorem~\ref{thm:classical-fillable}.

\appendix
\section{Proof of the Duality Theorem}
\label{sec:duality-proof}

We conclude this paper by proving Theorem~\ref{thm:duality}, repeating some of the proof of the original statement so as to properly set notation for the strengthening. The key ingredient in the proof of the duality exact sequence is Lemma 7.1 from \cite{josh-lisa:obstr}, which we repeat here for the reader's convenience:

\begin{lem}[\cite{josh-lisa:obstr}] \label{lem:+-iso} Assume $f: M^n \times \rr^N \to \rr$ is linear-at-infinity; let $\delta: M^n \times \rr^{2N} \to \rr$ be its
associated difference function.  For sufficiently large $\omega$ and 
for all $a \in \rr$, there is an isomorphism
$$\beta: H^j\left(\delta^\omega, \delta^{-a} \right) \stackrel{\sim}{\to} H_{2N+n-j}\left(\delta^{a}, \delta^{-\omega} \right).$$
\end{lem}

\begin{proof}[Proof of Theorem~\ref{thm:duality}(1)] The duality exact sequence is modeled on the cohomology long exact sequence of the triple $(\delta^\omega, \delta^{\epsilon},\delta^{-\epsilon})$ for sufficiently large $\omega$ and sufficiently small $\epsilon>0$; we have already encountered this triple in the long exact sequence (\ref{eq:gh-les}).  In fact, using this long exact sequence and the duality isomorphism from Lemma~\ref{lem:+-iso}, we obtain the following commutative diagram:
\begin{equation*} \label{eq:PAL}
  \xymatrix@C=12pt{
    \ar[r] & H^{k+N}(\delta^\omega, \delta^\epsilon) \ar[r]^{p_k} \ar[d]^\alpha_\simeq
    & H^{k+N}(\delta^\omega, \delta^{-\epsilon}) \ar[r]^{s_k} \ar[d]^\beta_\simeq
    & H^{k+N}(\delta^\epsilon, \delta^{-\epsilon}) \ar[d]^\gamma_\simeq \ar[r]^-{\delta_k}
    & \\
    \ar[r] & H_{n-k+N}(\delta^{-\epsilon}, \delta^{-\omega}) \ar[r] 
    & H_{n-k+N}(\delta^\epsilon, \delta^{-\omega}) \ar[r]^-{s'_{n-k}}
    & H_{n-k+N}(\delta^{\epsilon}, \delta^{-\epsilon}) \ar[r] &
  }
\end{equation*}

We now identify several terms in the  diagram above. By definition, $H^{k+N}(\delta^\omega, \delta^\epsilon) = GH^{k-1}(f)$; by the Thom isomorphism, $H^{k+N}(\delta^\epsilon, \delta^{-\epsilon}) \simeq H^k(\leg)$.  To identify the middle term on the bottom line, observe that the term $H_{n-k+N}(\delta^\epsilon, \delta^{-\omega})$ is part of the homology long exact sequence of $(\delta^\omega,\delta^\epsilon, \delta^{-\omega})$. Owing to the fact that $H_*(\delta^\omega, \delta^{-\omega})$ vanishes, the connecting homomorphism $\partial_{n-k}: H_{n-k+N+1}(\delta^\omega,\delta^\epsilon) \to H_{n-k+N}(\delta^\epsilon, \delta^{-\omega})$ is an isomorphism for all $k$.  Since, by definition,  $GH_{n-k}(f) =H_{n-k+N+1}(\delta^\omega,\delta^\epsilon)$, we see that the middle bottom term is isomorphic, via $\partial_{n-k}$, to  $GH_{n-k}(f)$.

The map $\rho_k: GH^{k-1}(f) \to GH_{n-k}(f)$ is defined as $\rho_k = \partial^{-1}_{n-k} \circ \beta \circ p_k$, and 
the map $\sigma_k: GH_{n-k}(f) \to H^k(\leg)$ is defined as  $\sigma_k = s_k \circ \beta^{-1} \circ \partial_{n-k}$.   
Since field coefficients are used, the adjoint  $\delta^*_{n-k}$ of $\delta_{n-k}$  fits into the following commutative diagram, obtained from the long exact sequences of the triples $(\delta^\omega, \delta^{\epsilon}, \delta^{-\epsilon})$ and 
$(\delta^\omega, \delta^{\epsilon}, \delta^{-\omega})$:
\begin{equation*} \label{eq:PAL2}
  \xymatrix@C=12pt{
    \ar[r] & H_{n-k+N+1}(\delta^\omega, \delta^\epsilon) \ar[r]^-{\delta_{n-k}^*}
    & H_{n-k+N}(\delta^\epsilon, \delta^{-\epsilon}) \ar[r] 
    & \\
    \ar[r] & H_{n-k+N+1}(\delta^{\omega}, \delta^{\epsilon}) \ar[u]_{\operatorname{id}} \ar[r]^-{\partial_{n-k}} 
    & H_{n-k+N}(\delta^\epsilon, \delta^{-\omega}) \ar[r] \ar[u]_{i_* = s'_{n-k}}
    & 
  }
\end{equation*}
Thus $\delta_{n-k}^* = \left(s'_{n-k}\right) \circ \partial_{n-k} = \left(\gamma \circ s_k \circ \beta^{-1}\right) \circ \partial_{n-k} = \gamma \circ \sigma_k$, as claimed.
\end{proof}

In order to prove the second part of Theorem~\ref{thm:duality}, namely the existence of the fundamental class, we need to use a slight extension of the standard Morse-Bott machinery described in \cite{austin-braam}.  We begin by recalling some of the basic objects in Morse-Bott theory to set notation.  Fix a generating family $f$ for $\leg$ and consider its difference function $\delta$ and a Morse-Bott-Smale metric on the domain $M \times \rr^{2N}$ of $\delta$. 
Given nondegenerate critical points $q$ and $r$ of $\delta$, we denote by $\mathcal{M}(q;r)$ the moduli space of gradient trajectories of $\delta$ from $r$ to $q$ modulo reparametrization by a constant shift. It is a smooth manifold of dimension $\ind_\delta q - \ind_\delta r -1$.
Let $\Sigma \subset M \times \Delta$ be the critical submanifold of $\delta$ in its zero level set. Similarly, we denote by $\mathcal{M}(q; \Sigma)$ the moduli space of gradient trajectories from $\Sigma$ to $q$ modulo reparametrization by a constant shift. It is a smooth manifold of dimension $\ind_\delta q - N - 1$. It is equipped with a smooth evaluation map $\textrm{ev} :  \mathcal{M}(q; \Sigma) \to \Sigma$ at $-\infty$. 

The key construction for the proof of Theorem~\ref{thm:duality}(2) is a new moduli space of gradient flow lines that begin on the ``diagonal'' $M \times \Delta \subset M \times \rr^N \times \rr^N$. Given a nondegenerate critical point $q$ of $\delta$, we define $\mathcal{M}(q; M \times \Delta)$ to be the space of gradient flow lines $\gamma: [0,\infty) \to M \times \rr^{2N}$ of $\delta$ with $\gamma(0) \in M \times \Delta$ and $\lim_{t \to \infty} \gamma(t) = q$.  Finally, we define $\mathcal{M}(\Sigma; M \times \Delta)$ to be the space of gradient flow lines $\gamma: [0,\infty) \to M \times \rr^{2N}$ of $\delta$ with $\gamma(0) \in M \times \Delta$ and $\lim_{t \to \infty} \gamma(t) \in \Sigma$.  The moduli space $M(\Sigma; M \times \Delta)$ consists entirely of constant trajectories, so  we may identify the moduli space with $\Sigma$ and its evaluation map at $+\infty$ with the identity.

\begin{lem} \label{lem:diag-compactness}
The moduli space $\mathcal{M}(q; M \times \Delta)$ is a smooth manifold of dimension $\ind_\delta q - N$.
If it has dimension $1$, it is compactified by the $0$-dimensional moduli spaces 
$$
\mathcal{M}(\Sigma; M \times \Delta)  \times_\Sigma \mathcal{M}(q; \Sigma)
\quad \textrm{and} \quad  
\bigcup_{\ind_\delta r = N} \mathcal{M}(r; M \times \Delta) \times \mathcal{M}(q; r)
$$
so that each element in these moduli spaces is in the closure of a single end of $\mathcal{M}(q; M \times \Delta)$.
\end{lem}

\begin{proof}
The moduli space $\mathcal{M}(q; M \times \Delta)$ is the intersection of the stable manifold of $q$ with 
$M \times \Delta$. After a small perturbation of the metric on $M \times \rr^{2N}$, this intersection is transverse, so that
$\mathcal{M}(q; M \times \Delta)$ is a smooth manifold of dimension $\ind_\delta q + (n +N) - (n+ 2N) = \ind_\delta q - N$.

The compactification of the space of gradient trajectories $\gamma: \rr \to M \times \rr^N \times \rr^N$ having their restriction to $[0, \infty)$ in $\mathcal{M}(q; M \times \Delta)$
consists of broken gradient trajectories~\cite[Lemma~3.3]{austin-braam}. Any such broken gradient trajectory contains a gradient 
trajectory passing through $M \times \Delta$ at $t=0$, so that its restriction to $[0, \infty)$ is either an element of
$\mathcal{M}(\Sigma; M \times \Delta)$ or an element of $\mathcal{M}(r; M \times \Delta)$ for some nondegenerate 
critical point $r$ of $\delta$. Note that the part of the broken trajectory passing through $M \times \Delta$ cannot be a piece of a trajectory from a critical point with negative critical value to $q$ since the image of the whole broken trajectory is symmetric under the exchange of the two $\rr^N$ factors.  The compactification of  $\mathcal{M}(q; M \times \Delta)$ therefore consists of broken trajectories with such an element 
followed by an ordinary broken gradient trajectory from $\Sigma$ or $r$ to $q$.

Let us first assume that the gradient trajectory passing through $M \times \Delta$ at $t=0$ is an element of 
$\mathcal{M}(\Sigma; M \times \Delta)$. 
The next gradient trajectory in the broken gradient trajectory runs from $\Sigma$ to a critical point $s$ of $\delta$ with
$\delta(s)> 0$, an element of the moduli space $\mathcal{M}(s;\Sigma)$ with dimension $\ind_\delta s - N - 1 \ge 0$. 
If $\ind_\delta q = N + 1$, this forces $s = q$ so that the broken trajectory is an element of
$\mathcal{M}(\Sigma; M \times \Delta)  \times_\Sigma \mathcal{M}(q; \Sigma)$. Conversely, the Morse-Bott gluing 
theorem~\cite[Theorem~A.11]{austin-braam} shows that any element of this moduli space  is the limit of a unique 
end of $\mathcal{M}(q; M \times \Delta)$.

We now assume that the gradient trajectory passing through $M \times \Delta$ at $t=0$ is an element of 
$\mathcal{M}(r; M \times \Delta)$.
If $\ind_\delta q = N + 1$, there can only be one more gradient trajectory in the above broken gradient trajectory: 
an element of $\mathcal{M}(q;r)$ with $\ind_\delta r = N$. Conversely, the usual gluing theorem in Morse theory shows 
that each element of $\mathcal{M}(r; M \times \Delta) \times \mathcal{M}(q; r)$ is the limit of a unique end of 
$\mathcal{M}(q; M \times \Delta)$.
\end{proof}

With this technical lemma in hand, we are ready for the proof of the existence of a fundamental class.

\begin{proof}[Proof of Theorem~\ref{thm:duality}(2)]
Let $\leg_1, \ldots, \leg_k$ be the connected components of $\leg$ and let $p_i$ be the generator of $H^0(\leg_i)$ for $i = 1, \ldots, k$. We claim that $p = \sum_{i=1}^k p_i \in \ker \delta_0$. In view of the property (1) proved above, this means that $\delta_n$ has rank at least $1$. 

Note that $\delta_0(p)$ results from the count of gradient trajectories with cascades~\cite{bh:morse-bott-cascades} from a minimum 
of an auxiliary Morse function on $\Sigma$ to a critical point $q$ of the difference function with $\ind q = N+1$. Such cascade gradient trajectories are in bijective correspondence with elements of $\mathcal{M}(q; \Sigma)$.

In order to prove our claim, define the cochain $c \in C^N(\delta^\omega, \delta^\epsilon)$ by:
\[c = \sum_{\textrm{ind }r = N} \# \mathcal{M}(r; M \times \Delta) \, r.\]
Denoting the Morse codifferential of $c$ by $dc$, the claim above follows from the assertion that $\delta_0(p) + dc = 0$.  Lemma~\ref{lem:diag-compactness} implies that this relation encodes a count of the boundary of the one-dimensional moduli space $\bigcup_{\ind q = N+1} \mathcal{M}(q; M \times \Delta)$: $\delta_0(p)$ accounts for the first type of broken flow line in the lemma,
while $dc$ accounts for the second type.  
\end{proof}


\bibliographystyle{amsplain} \bibliography{main}

\end{document}